\documentclass{article}

     \usepackage[final,nonatbib]{neurips_2021}

\usepackage[utf8]{inputenc} 
\usepackage[T1]{fontenc}    
\usepackage{hyperref}       
\usepackage{url}            
\usepackage{booktabs}       
\usepackage{amsfonts}       
\usepackage{nicefrac}       
\usepackage{microtype}      

 \usepackage{algorithm, algorithmic}
\usepackage{fancyhdr}
\usepackage{amsmath}
\usepackage{amsthm}
\usepackage{mathtools}
\usepackage{amssymb,url}
\usepackage{ragged2e}
 \usepackage{amscd}
\usepackage{graphicx}
 \usepackage{multirow}
 \usepackage{enumerate}
 \usepackage{float}
 \usepackage{amsthm}
 \usepackage[]{amsrefs}
 \usepackage{bm}
 \usepackage{caption}
 \usepackage{subcaption}
\usepackage{multicol}
  \usepackage{cleveref}
\usepackage[dvipsnames]{xcolor}
 
\usepackage{etoc}

 \usepackage{tikz} 
\usepackage{pgfplots}
\pgfplotsset{compat=newest}
\usetikzlibrary{plotmarks}
\usetikzlibrary{arrows.meta}
\usepgfplotslibrary{patchplots}
\newcommand{\sigmaxstar}{ \sigma_s \left(  x_{\ast } \right)_{\ell_1}}

\pgfplotsset{plot coordinates/math parser=false}
\newlength\figureheight
\newlength\figurewidth

\DeclareMathOperator{\Id}{Id}

\DeclareMathOperator*{\argmin}{arg\,min}

\DeclareMathOperator{\sgn}{sgn}

\DeclareMathOperator{\N}{\mathbb{N}}
\DeclareMathOperator{\R}{\mathbb{R}}

\DeclareMathOperator{\diag}{diag}

\newtheorem{theorem}{Theorem}[section]
\newtheorem{lemma}[theorem]{Lemma}
\newtheorem{proposition}[theorem]{Proposition}

\newtheorem{remark}[theorem]{Remark}

\newtheorem{definition}[theorem]{Definition}

\newcommand\f[1]{\mathbf{#1}}

\newcommand{\onenorm}[1]{\Vert #1 \Vert_1}
\newcommand{\twonorm}[1]{\Vert #1 \Vert_2}
\newcommand{\rhos}{ \rho_s}

\newcommand*\samethanks[1][\value{footnote}]{\footnotemark[#1]}

\title{Iteratively Reweighted Least Squares for Basis Pursuit with Global Linear Convergence Rate}

\author{Christian K\"ummerle\thanks{\href{mailto:kuemmerle@jhu.edu}{kuemmerle@jhu.edu}} \samethanks[4]  \\ 
Department of Applied Mathematics \& Statistics \\ Johns Hopkins University
\And
Claudio Mayrink Verdun\thanks{\href{mailto:verdun@ma.tum.de}{verdun@ma.tum.de}} \samethanks[4]\\
Department of Mathematics and Department of Electrical and Computer Engineering \\
Technical University of Munich
\And
 Dominik St\"oger\thanks{\href{mailto:dominik.stoeger@ku.de}{dominik.stoeger@ku.de}} \thanks{The name order of the authors is alphabetical.} \\
 Department of Mathematics \\ KU Eichstätt–Ingolstadt
 }

\begin{document}

\maketitle

\etocdepthtag.toc{mtchapter}
\etocsettagdepth{mtchapter}{subsection}
\etocsettagdepth{mtappendix}{none}

\begin{abstract}
The recovery of sparse data is at the core of many applications in machine learning and signal processing. While such problems can be tackled using $\ell_1$-regularization as in the LASSO estimator and in the Basis Pursuit approach, specialized algorithms are typically required to solve the corresponding high-dimensional non-smooth optimization for large instances.
Iteratively Reweighted Least Squares (IRLS) is a widely used algorithm for this purpose due its excellent numerical performance. 
However, while existing theory is able to guarantee convergence of this algorithm to the minimizer, it does not provide a global convergence rate. 
In this paper, we prove that a variant of IRLS converges \emph{with a global linear rate} to a sparse solution, i.e., with a linear error decrease occurring immediately from any initialization, if the measurements fulfill the usual null space property assumption. We support our theory by numerical experiments showing that our linear rate captures the correct dimension dependence. We anticipate that our theoretical findings will lead to new insights for many other use cases of the IRLS algorithm, such as in low-rank matrix recovery.
\end{abstract}

\section{Introduction} \label{sec:intro}

The field of sparse recovery deals with the problem of recovering an (approximately) sparse vector $x$ from only few linear measurements, presented by an underdetermined system of linear equations of the form $y=Ax$. One approach to solve this problem is to consider the $\ell_0$-minimization under linear constraints, which is NP-hard in general \cite{natarajan,davis_mallat_avellaneda}. 
For computational reasons, instead of $\ell_0$-minimization, it is common practice to consider its convex relaxation
\begin{equation}\label{basis_pursuit}
\min_{x\in \mathbb{R}^N }||x||_1 \qquad \text{subject to}  \ Ax=y,
\tag{$P_1$}
\end{equation}

where $A \in \mathbb{R}^{m \times N}$, $b \in \mathbb{R}^m$ are given, which is referred to as \emph{$\ell_1$(-norm) minimization} \cite{TaylorBankMcCoy79,claerbout_muir,DonohoLogan92} or \emph{basis pursuit} \cite{chen_donoho,ChenDonohoSaunders01} in the literature. 

Unlike  $\ell_0$-minimization, the optimization program \eqref{basis_pursuit} is computationally tractable in general, and a close relationship of their minimizers has been recognized and well-studied in the theory of compressive sensing \cite{CT06,CRT06b,donoho,FoucartRauhut13}. In statistics and machine learning, an unconstrained variant of \eqref{basis_pursuit}, often called \emph{LASSO}, amounts to the most well-studied tractable estimator for variable selection in high-dimensional inference \cite{tibshirani1996regression,hastie2019statistical,meinshausen2006high}. $\ell_1$-minimization has many other applications and it was even called the \emph{modern least squares} \cite{candes2008enhancing}.

The tractability of \eqref{basis_pursuit} becomes evident from the fact that it can be reformulated as a linear program \cite{tillmann2015equivalence}. However, as many problems of interest in applications are high-dimensional and therefore challenging for standard linear programming methods, many specialized solvers for \eqref{basis_pursuit} have been proposed, such as the Homotopy Method \cite{donoho_tsaig}, primal-dual methods \cite{chambolle2011first,molinari2021iterative}, Alternating Direction Method of Multipliers \cite{boyd_admm}, Bregman iterative regularization \cite{yin_osher_goldfarb_darbon} and Semismooth Newton Augmented Lagragian Methods \cite{li_sun_toh} and Iteratively Reweighted Least Squares (IRLS), the latter of which is in the focus of this paper. 
 
Iteratively Reweighted Least Squares corresponds to a family of algorithms that minimizes non-smooth objective functions by solving a sequence of quadratic problems, with its idea going back to a method proposed by Weiszfeld for the \emph{Fermat-Weber problem} \cite{weiszfeld,back_sabach}.  
A variety of different problems such as robust regression in statistics \cite{holland_welsch,bhaskar_govind_prateek_purushottam}, total variation regularization in image processing \cite{GemanReynolds92,NikolavaNg05,allain_idier_goussard}, joint learning of neural networks \cite{zhang2019joint}, robust subspace recovery \cite{lerman2018fast} and the recovery of low-rank matrices \cite{mohan_fazel,fornasier_rauhut_ward,kummerle_sigl,KM21} can be solved efficiently by IRLS in practice, as it relies on simple linear algebra to solve the linear systems arising from the quadratic problems at each iteration, without the need of a careful initialization or intricate parameter tuning. On the other hand, the analysis of IRLS methods is typically challenging: General convergence results are often weak, and stronger convergence results are only available in particular cases; see \Cref{sec:related:work} for more details.

\textbf{IRLS for sparse recovery.} In the sparse recovery context, the first variants of IRLS were introduced in \cite{rao_kreutz-delgado,gorodnitsky_rao} for the $\ell_p$-quasinorm minimization problem $(P_{p})$ with $0 <p \leq 1$ that is similar to \eqref{basis_pursuit}, but with $\|x\|_p$ instead of $\|x\|_1$ as an objective.
In \cite{chartrand_yin}, modifications of the method of \cite{rao_kreutz-delgado,gorodnitsky_rao} using specific smoothing parameter update rules were observed to exhibit excellent numerical performance for solving $(P_{p})$, retrieving the underlying sparse vector when most of the methods fail. A useful fact is that IRLS is one of the few methods (ADMM being the other one \cite{boyd_admm}) that provides a framework to solve both constrained and unconstrained formulations of $\ell_p$-minimization problems. 
 
A major step forward in the theoretical understanding of IRLS was achieved in \cite{Daubechies10}, where the authors showed that a variant of IRLS for \eqref{basis_pursuit} converges globally to an $\ell_1$-minimizer if  the measurement operator $A$ fulfills the null space property of sufficient order, which essentially ensures that an $\ell_1$-minimizer is actually sparse. However, since this proof relies on the existence of a convergent subsequence, their proof does not reveal any rate for \emph{global} convergence. The analysis of \cite{Daubechies10} provides, furthermore, a \emph{locally} linear convergence rate, but this local linear rate has the drawback that it only applies \emph{if the support of the true signal has been discovered}, which is arguably the difficult part of $\ell_0$-minimization---cf. \Cref{local_thm_daubechies} below and \Cref{section:numerical:localglobal}.

While several extensions and modifications of the IRLS algorithm in \cite{Daubechies10} have been proposed (see, e.g., \cite{AravkinBurkeHe19,fornasier_peter_rauhut_worm}), this following fundamental algorithmic question has remained unanswered:

\centering 
\textit{What is the global convergence rate of the IRLS algorithm for $\ell_1$-minimization?}
\vspace{-.2cm}
\justify
\textbf{Our contribution.} We resolve this question, formally stated in \cite{straszak2021iteratively}, and present a new IRLS algorithm that \emph{converges linearly} to a sparse ground truth, \emph{starting from any initialization}, as stated in \Cref{mainresult:sparse}. Our algorithm returns a feasible solution with $\delta$-accuracy, i.e., $ \onenorm{ x_{*} -x^k }\le \delta$, where $ x_{*}$ is the underlying $s$-sparse vector, in $k=O(N \sqrt{(\log N)/m} \log (1/\delta))$ iterations. Analogous to \cite{Daubechies10}, it is assumed that the measurement matrix $A$ satisfies the so-called null space property \cite{cohen_dahmen_devore}. We also provide a similar result for approximately sparse vectors. Our proof relies on a novel quantification of the descent of a carefully chosen objective function in the direction of the ground truth. Additionally, we support the theoretical claims by numerical simulations indicating that we capture the correct dimension dependence.
We believe that the new analysis techniques in this paper are of independent interest and will pave the way for establishing global convergence rates for other variants of IRLS such as in low-rank matrix recovery \cite{fornasier_rauhut_ward}.

\textbf{Notation.} We denote the cardinality of a set $I$ by $|I|$ and the support of a vector $x \in \mathbb{R}^N$, i.e., the index set of its nonzero entries, by $ \text{supp}(x)=\{j \in [N] : x_j \neq 0\}$. We call a vector $s$-sparse if at most $s$ of its entries are nonzero. We denote by $x_I$ the restriction of $x$ onto the coordinates indexed by $I$, and use the notation $I^c := [N] \setminus I $ to denote the complement of a set $I$. Furthermore, $\sigma_s(x)_{\ell_1}$ denotes the $\ell_1$\emph{-error of  the best $s$-term approximation} of a vector $x \in \mathbb{R}^N$, i.e., $\sigma_s(x)_{\ell_1}=\inf \{\Vert x-z \Vert_1: \ z \in \mathbb{R}^N \ \text{is $s$-sparse} \}$.
 \vspace{-.4cm}
\section{IRLS for sparse recovery}\label{section:background}
We now present a simple derivation of the Iteratively Reweighted Least Squares (IRLS) algorithm for $\ell_1$-minimization which is studied in this paper. IRLS algorithms can be interpreted as a variant of a Majorize-Minimize (MM) algorithm \cite{sun_babu_palomar}, as we will lay out in the following. It mitigates the non-smoothness of the $\onenorm{\cdot}$-norm by using the smoothed objective function $\mathcal{J}_\varepsilon: \R^N \to \R$, which is defined, for a given $\varepsilon > 0$, by
\begin{equation} \label{eq:smoothedell1:objective}
\mathcal{J}_{\varepsilon}(x) := \sum_{i=1}^N j_{\varepsilon}(x_i) \quad \text{ with } \quad j_{\varepsilon}(x) := \begin{cases}
 	|x|, & \text{ if } |x| > \varepsilon, \\
 	\frac{1}{2}\left(\frac{x^2}{\varepsilon} + \varepsilon \right), & \text{ if } |x| \leq \varepsilon.
 \end{cases}
\end{equation}
The function $\mathcal{J}_{\varepsilon}$ can be considered as a scaled Huber loss function which is widely used in robust regression analysis \cite{Huber64,meyer2021alternative}. 
Moreover, the function $ \mathcal{J}_{\varepsilon} $ is continuously differentiable and fulfills $|x|\leq j_{\varepsilon}(x) \leq |x|+\varepsilon$ for each $x \in \R$. Instead of minimizing the function $\mathcal{J}_{\varepsilon}$ directly, the idea of IRLS is to minimize instead a suitable chosen quadratic function $Q_{\varepsilon}(\cdot,x)$, which majorizes $\mathcal{J}_{\varepsilon}$ such that  $Q_{\varepsilon} \left(z,x\right) \ge \mathcal{J}_{\varepsilon} \left(z\right) $ for all $z \in \mathbb{R}^N$. This function is furthermore chosen such that $Q_{\varepsilon} \left(x,x\right) = \mathcal{J}_{\varepsilon} \left( x\right) $ holds, which implies that $ \underset{z \in \mathbb{R}^n}{\min} \ Q_{\varepsilon} \left( z,x \right) \le \mathcal{J}_{\varepsilon} \left(x\right)$. The latter inequality implies that by minimizing $Q_{\varepsilon}(\cdot,x)$, IRLS actually achieves an improvement in the value of $\mathcal{J}_{\varepsilon}$ as well. More specifically, $Q_{\varepsilon}\left(\cdot, x \right)$ is defined by
\begin{equation} \label{eq:smoothedell1:IRLSmajorizer}
\begin{split}
Q_{\varepsilon}(z,x) &:= \mathcal{J}_{\varepsilon}(x)  + \langle \nabla \mathcal{J}_{\varepsilon}(x), z-x\rangle + \frac{1}{2} \langle (z-x), \diag(w_{\varepsilon}(x)) (z-x) \rangle \\
 &= \mathcal{J}_{\varepsilon}(x)  +\frac{1}{2} \langle z, \diag(w_{\varepsilon}(x)) z \rangle - \frac{1}{2} \langle x, \diag(w_{\varepsilon}(x)) x \rangle,
\end{split}
\end{equation}
where $\nabla \mathcal{J}_{\varepsilon}(x) = \left(\begin{cases}
	\frac{x_i}{|x_i|}, & \text{ if } |x_i| > \varepsilon \\
	\frac{x_i}{\varepsilon}, & \text{ if } |x_i| \leq \varepsilon
	\end{cases}\right)_{i=1}^{N}$
is the gradient of $\mathcal{J}_{\varepsilon}$ at $x$ and the weight vector $ w_{\varepsilon} \left(x\right) \in \R^N$ is a vector of \emph{weights} such that $ w_{\varepsilon}(x)_i := [ \max(|x_i|,\varepsilon)]^{-1} \quad  \text{for } i \in [N]$.
The following lemma shows that $Q_{\varepsilon} \left(\cdot, \cdot \right) $ has indeed the above-mentioned properties. We refer to the supplementary material for a proof.
\begin{lemma} \label{lemma:IRLS:basicproperties}
	Let $\varepsilon >0$, let  $\mathcal{J}_{\varepsilon}: \R^N  \to \R$ be defined as in \eqref{eq:smoothedell1:objective} and $Q_{\varepsilon}: \R^N \times \R^N \to \R$ as defined in \eqref{eq:smoothedell1:IRLSmajorizer}. Then, for any  $z,x \in \R^N$, the following affirmations hold:

\begin{itemize}
\begin{minipage}{0.36\linewidth}
    \item[i.] $\diag(w_{\varepsilon}(x) ) x = \nabla \mathcal{J}_{\varepsilon}(x) $, 
\end{minipage}
\begin{minipage}{0.32\linewidth}
   \item[ii.] $ Q_{\varepsilon}(x,x) = \mathcal{J}_{\varepsilon}(x)$,
\end{minipage}
\begin{minipage}{0.3\linewidth}
    \item[iii.] $Q_{\varepsilon}(z,x) \geq \mathcal{J}_{\varepsilon}(z)$.
\end{minipage}
\end{itemize}

\end{lemma}

As can be seen from the equality in \eqref{eq:smoothedell1:IRLSmajorizer}, minimizing of $Q_{\varepsilon}(\cdot,x)$ corresponds to a minimizing \emph{\mbox{(re-)weighted} least squares objective} $\langle \cdot, \diag(w_{\varepsilon}(x)) \cdot \rangle$, which lends its name to the method. Note that unlike a classical MM approach, however, IRLS comes with an \emph{update} step \emph{of the smoothing parameter} $\varepsilon$ at each iteration. We provide an outline of the method in \Cref{alg:algo1}. 

\begin{algorithm}
\caption{Iteratively Reweighted Least Squares for $\ell_1$-minimization}
\label{alg:algo1}
\begin{algorithmic}
\STATE{\textbf{Input:} Measurement matrix $A \in \mathbb{R}^{m \times N}$, data vector $y \in \mathbb{R}^m$,  \\ initial weight vector $w_0 \in \R^N$ (default: $w_0 = (1,1,\ldots,1)$).}
\STATE{Set $\varepsilon_0 = \infty$.}
\FOR{$k = 0, 1, 2,\ldots$}
\STATE{ \vspace*{-7mm}	\begin{align}
 			 x^{k+1} &:= \argmin_{z \in \R^N}  \  \langle z, \diag \left(w_{k}\right)  z  \rangle \; \text{ subject to } \; A z = y, \label{step_1} \\ 
 		   	\varepsilon_{k+1}&: =  \min\left(\varepsilon_k , \frac{ \sigma_s(x^{k+1})_{\ell_1}}{N} \right), \label{step_2} \\
 		   	\left( w_{k+1} \right)_i &:=  \frac{1}{ \max \left( \vert x_i^{k+1}  \vert , \varepsilon_{k+1}   \right) }  \quad \quad \text{ for each } i \in [N],\label{eq:IRLS:step3}%
 	\end{align}  \vspace*{-7mm}	}
\ENDFOR
\RETURN Sequence $(x^k)_{k\geq 1} $.
\end{algorithmic}
\end{algorithm}

The weighted least squares update \eqref{step_1} can be computed such that $x^{k+1} = W_k^{-1} A^* (A W_k^{-1} A^*)^{-1} (y)$ with $W_k = \diag \left(w_{k}\right)$, with the solution of the $(m \times m)$ linear system $(A W_k^{-1} A^*) z =  y$ as a main computational step. This linear system is positive definite and suitable for the use of iterative solvers. In \cite{fornasier_peter_rauhut_worm}, an analysis of how accurately the linear system of a similar IRLS method needs to be solved to ensure overall convergence. We note that for small $\varepsilon_k$, the Sherman-Woodbury formula \cite{Woodbury50} can be used so that the calculation of $x^{k+1}$ boils down to solving a smaller linear system that is well-conditioned, c.f. the supplementary material for details. This numerically advantageous property is not shared by the methods of \cite{Daubechies10,fornasier_peter_rauhut_worm,AravkinBurkeHe19}, as our smoothing update \eqref{step_2} is slightly different from the ones proposed in these papers. We refer to \Cref{sec:existing_theory} for a discussion.

The update step of the smoothing parameter $\varepsilon$ \eqref{step_2} for the IRLS algorithm under consideration requires an a priori estimate of the sparsity of the ground truth of the signal, a piece of information that is also needed by most of the methods for sparse reconstruction. In practice, an overestimation of $s$ is not a problem for similar numerical results if the overestimation remains within small multiples of the sparsity of the signal. We note, however, that there are also versions of IRLS which do not require a-priori knowledge of $s$, e.g. \cite{voronin_daubechies,fornasier_peter_rauhut_worm}, as the update rule for the smoothing parameter is chosen differently. An interesting future research direction is to extend the analysis presented here to IRLS with such a smoothing parameter update.

A consequence of \Cref{lemma:IRLS:basicproperties}, step \eqref{step_2}, the fact that $\varepsilon \mapsto \mathcal{J}_{\varepsilon}(z)$ is monotonously non-decreasing, and that $l \mapsto \varepsilon_k$ is non-increasing is that $k\mapsto \mathcal{J}_{\varepsilon}(z)$ is non-increasing in k. This implies that the iterates $x^{k},x^{k+1}$ of \Cref{alg:algo1} fulfill
\begin{equation} \label{eq:J:monotonicity:1}
 \mathcal{J}_{\varepsilon_{k+1}}(x^{k+1}) \leq \mathcal{J}_{\varepsilon_{k}}(x^{k+1}) \leq Q_{\varepsilon_k}(x^{k+1},x^k) \leq Q_{\varepsilon_k}(x^{k},x^k) = \mathcal{J}_{\varepsilon_{k}}(x^{k}).
\end{equation}
This shows in particular that the sequence $ \left\{ \mathcal{J}_{\varepsilon_k} \left(x^k\right) \right\}_{k=0}^{\infty} $ is non-increasing. For this reason, it can be shown that each accumulation point of the sequence of iterates $(x^k)_{k\geq 0}$ is a (first-order) stationary point of the smoothed $\ell_1$-objective $J_{\overline{\varepsilon}}(\cdot)$ subject to the measurement constraint imposed by $A$ and $y$, where $\overline{\varepsilon} = \lim_{k \to \infty} \varepsilon_k$  (see \cite[Theorem 5.3]{Daubechies10}).

\subsection{Null space property}
As in \cite{Daubechies10}, the analysis we present is based on the assumption that the measurement matrix $A$ satisfies the so-called null space property \cite{cohen_dahmen_devore,gribonval_nielsen}, which is a key concept in the compressed sensing literature (see, e.g., \cite[Chapter 4]{FoucartRauhut13} for an overview).
\begin{definition} \label{def:NSP:statement}
	\label{def_NSP} A matrix $A \in \mathbb{R}^{m \times N}$ is said to satisfy the \emph{$\ell_1$-null space property ($\ell_1$-NSP)} of order $s \in \N$ with constant $0 < \rho_s < 1$ if for any set $S \subset [N]$ of cardinality $|S| \leq s$, it holds that $ \onenorm{v_S} \leq \rho_s  \onenorm{v_{S^c}}$, for all  $ v \in \ker(A)$.
\end{definition}	

In \cite[Chapter 4]{FoucartRauhut13}, the property of \Cref{def:NSP:statement} was called \emph{stable} null space property.
The importance of the null space property is due to the fact that it gives a necessary and sufficient criterion for the success of basis pursuit for sparse recovery, as the following theorem shows.
\begin{theorem}[{\cite[Theorem 4.5]{FoucartRauhut13}}] 
	Given a matrix $A \in \mathbb{R}^{m \times N}$, every vector $ x\in \mathbb{R}^N$ such that $||x||_0\leq s$ is the unique solution of \eqref{basis_pursuit} with $Ax=y$ if and only if $A$ satisfies the null space property of order $s$ for some $0 < \rho_s < 1$.
\end{theorem}
The $\ell_1$-NSP is implied by the restricted isometry property (see, e.g., \cite{cahil_chen_wang}), which is fulfilled by a large class of random matrices with high probability. For example, this includes matrices with (sub-)gaussian entries and random partial Fourier matrices \cite{rudelson_vershynin, baraniuk_davenport_devore_wakin}. 

\subsection{Existing theory} \label{sec:existing_theory}

A predecessor of IRLS for the sparse recovery problem \eqref{basis_pursuit}, and more generally, for $\ell_p$-quasinorm minimization with $0 < p \leq 1$, is the \emph{FOCal Underdetermined System Solver} (FOCUSS) as proposed by Gorodnitsky, Rao and Kreutz-Delgado \cite{gorodnitsky_rao, rao_kreutz-delgado}. Asymptotic convergence of FOCUSS to a stationary point from any initialization was claimed in \cite{rao_kreutz-delgado}, but the proof was not entirely accurate, as pointed out by \cite{Chartrand2016}.
One limitation of FOCUSS is that, unlike in IRLS as presented in \Cref{alg:algo1}, no smoothing parameter $\varepsilon$ is used, which leads to ill-conditioned linear systems. 

To mitigate this, \cite{chartrand_yin} proposed an IRLS method that uses smoothing parameters $\varepsilon$ (such as used in $Q_{\varepsilon}$ defined above) that are updated iteratively.
It was observed that this leads to a better condition number for the linear systems to be solved in each step of IRLS and, furthermore, that this smoothing strategy has the advantage of finding sparser vectors if the weights of IRLS are chosen to minimize a non-convex $\ell_p$-quasinorm for $p<1$.

Further progress for IRLS designed to minimize an $\ell_1$-norm was achieved in the seminal paper \cite{Daubechies10}. In \cite{Daubechies10}, it was shown that if the measurement operator fulfills a suitable $\ell_1$-null space property as in \Cref{def_NSP}, an IRLS method with iteratively updated smoothing converges to an $s$-sparse solution, coinciding with the $\ell_1$-minimizer, if there exists one that is compatible with the measurements. This method uses not exactly the update rule of \eqref{step_2}, but rather updates the smoothing parameter such that $\varepsilon_{k+1}=\min (\varepsilon_k,R(x^{k+1})_{s+1}/N)$, where $R(x^{k+1})_{s+1}$ is the $(s+1)\textsuperscript{st}$-largest element of the set $\{ |x_j^{k+1}|, j \in [N]\}$. Furthermore, a \emph{local linear convergence rate} of IRLS was established \cite[Theorem 6.1]{Daubechies10} under same conditions.

However, the analysis of \cite{Daubechies10} has its limitations: First, there is a gap in the assumption of their convergence results between the sparsity $s$ of a vector to be recovered and the order $\widehat{s}$ of the NSP of the measurement operator. 
Recently, this gap was circumvented in \cite{AravkinBurkeHe19} with an IRLS algorithm that uses a smoothing update rule based on an $\ell_1$-norm, namely, $\varepsilon_{k+1}=\min (\varepsilon_k,\eta (1-\rho_s)\sigma_{s}(x^{k+1})_{\ell_1}/N)$, where $\eta \in (0,1)$, and $\rho_s$ is the NSP constant of the order $s$ of the NSP fulfilled by the measurement matrix $A$---this rule is quite similar to the rule \eqref{step_2} that we use in \Cref{alg:algo1}. In particular, \cite[Theorem III.6]{AravkinBurkeHe19} establishes convergence with local linear rate similar to \cite{Daubechies10} without the gap mentioned above. The main limitation, however, of the theory of \cite{Daubechies10} (which is shared by \cite{AravkinBurkeHe19}) is that the linear convergence rate only holds \emph{locally}, i.e., in a situation where the support of the sparse vector has already been identified, see also \Cref{sec:main:results} and \Cref{section:numerical:localglobal} for a discussion.

We finally mention three relevant papers for the theoretical understanding of IRLS. \cite{ba2013convergence} established the correspondence between the IRLS algorithms and the Expectation-Maximization algorithm for constrained maximum likelihood estimation under a Gaussian scale mixture distribution. By doing so, they established similar results as those from \cite{Daubechies10}, i.e., the global convergence of IRLS with local linear convergence rate (as can be seen from their equation (38), which similar to \eqref{eq:DDFG:closeness:assumption} below) but by using different techniques based on such correspondence. \cite{straszak2021iteratively} explores the relationship of IRLS for $\ell_1$-minimization and a slime mold dynamics, interpreting both as an instance of the same meta-algorithm. Without requiring any connection between sparse recovery and $\ell_1$-minimization, \cite{ene2019improved} shows that an IRLS-like algorithm for \eqref{basis_pursuit}, requires $O(N^{1/3} \log(1/\delta)/\delta^{2/3}+ \log(N)/\delta^2)$ iterations to obtain a multiplicative error of $1+\delta$ on the minimizer $||x||_1$. Unlike our result  \Cref{mainresult:sparse}, this corresponds not to a linear, but to a sublinear convergence rate.

\subsection{Related work} \label{sec:related:work}

As mentioned in the introduction, IRLS has a long history and has appeared under different names within different communities, e.g., similar algorithms are usually called \emph{half-quadratic algorithms} in image processing \cite{idier,allain_idier_goussard} and the \emph{Ka\u{c}anov method} in numerical PDEs \cite{diening_fornasier_tomasi_wank}. Probably the most common usage of IRLS has been in robust regression \cite{holland_welsch,green1984iteratively}, c.f. \cite{burrus} for a survey that also covers applications in approximation theory. For $p$-norm regression, \cite{adil_peng_sachdeva} proposed a version of IRLS for which convergence results for $ p \in [2,\infty)$ were established, solving a problem that was open for over thirty years. Also, for robust regression, by using an $\ell_1$-objective on the residual, \cite{bhaskar_govind_prateek_purushottam} showed recently global convergence of IRLS with a linear rate, with high probability for sub-Gaussian data. We note that our proof strategy is different from the one of \cite{bhaskar_govind_prateek_purushottam} due to a structural difference of \eqref{basis_pursuit} from robust regression.
 
In \cite{ochs_dosovitskiy_brox_pock}, the authors provide a general framework for formulating IRLS algorithms for the optimization of a quite general class of non-convex and non-smooth functions, however, without updated smoothing. They use techniques developed in \cite{attouch_bolte_svaiter} to show convergence of the sequence of iterates to a critical point under the Kurdyka-\L{}ojasiewicz property \cite{Bolte2007lojasiewicz}. However, no results about convergence rates were presented.

For the sparse recovery problem, the topic discussed in this paper, the references \cite{lai_yangyang_wotao, fornasier_peter_rauhut_worm, voronin_daubechies} analyzed IRLS for an unconstrained version of \eqref{basis_pursuit}, which is usually a preferable formulation if the measurements are corrupted by noise. Additionally, the work \cite{fornasier_peter_rauhut_worm} addressed the question of how to solve the successive quadratic optimization problems. The authors developed a theory that shows, under the NSP, how accurately the quadratic subproblems need to be solved via the conjugate gradient method in order to preserve the convergence results established in \cite{Daubechies10}.

Finally, for the related problems of low-rank matrix recovery and completion, IRLS strategies have emerged as one of the most successful methods in terms of data-efficiency and scalability \cite{fornasier_rauhut_ward,mohan_fazel,kummerle_sigl,KM21}. 

While we were writing this paper, the manuscript \cite{2021-Poon-noncvxpro} appeared providing new insights about IRLS. It describes a surprisingly simple reparametrization of the IRLS formulation for $\ell_p$-minimization (with $p \in (2/3,1)$) that leads to a smooth bilevel optimization problem without any spurious minima, i.e., the stationary points of this new formulation are either global minima or strict saddles. It is an interesting future direction to explore the connection between this new approach and our global convergence theory.

\section{IRLS for Basis Pursuit with Global Linear Rate} \label{sec:main:results}

As discussed in \Cref{sec:existing_theory}, the main theoretical advancements for IRLS for the sparse recovery problem were achieved in the work \cite{Daubechies10}. 

\begin{proposition}\cite[Theorem 6.1]{Daubechies10}\label{local_thm_daubechies}
Assume that $A \in \mathbb{R}^{m \times N}$ satisfies the NSP of order $\widehat{s} >s$ with constant $\rho_{\widehat{s}}$ such that $ 0< \rho_{\widehat{s}} < 1- \frac{2}{\widehat{s}+2}$ and $\hat{s} > s + \frac{2\rho_{\widehat{s}} }{1- \rho_{\widehat{s}}} $ hold. Let $x_* \in \mathbb{R}^N$ be an $s$-sparse vector and set $y=Ax_*$.
Assume that there exists an integer $k_0 \ge 1$ and a positive number $\xi > 0$ such that
\begin{equation} \label{eq:DDFG:closeness:assumption}
\xi:= \frac{\|x^{k_0} - x_{*}\|_1}{\min_{ i \in S} |(x_*)_{i}|} < 1.
\end{equation}
 Then the iterates $\{x^{k_0},x^{k_0+1},x^{k_0+2},\ldots\}$ of the IRLS method in \cite{Daubechies10} converge \emph{linearly} to $x_*$, i.e., for all $k\geq k_0$, the $k$th iteration of IRLS satisfies
\begin{equation} \label{eq:DDFG:localrate}
\Vert x^{k+1}-x_*\Vert_1 \leq \frac{\rho_{\widehat{s}}(1+\rho_{\widehat{s}})}{1-\xi}\left( 1+ \frac{1}{\widehat{s}-1-s}\right) \Vert x^{k}-x_*\Vert_1.
\end{equation}
\end{proposition}
The main contribution of this paper is that we overcome a local assumption such as \eqref{eq:DDFG:closeness:assumption} and show that IRLS as defined by \Cref{alg:algo1} \emph{exhibits a global linear convergence rate}, i.e., there is a linear convergence rate starting from any initialization, as early as in the first iteration.

\textbf{Exactly sparse case.} Our first main result, \Cref{mainresult:sparse}, deals with the scenario that the ground truth vector $x_{*}$ is exactly $s$-sparse. Our second result, presented in the supplementary material, generalizes the first one to the more realistic situation of approximately sparse vectors.

\begin{theorem}\label{mainresult:sparse}
Consider the problem of recovering an unknown $s$-sparse vector $ x_{*} \in \R^N $ from known measurements of the form $y=Ax_{*} $. Assume that the measurement matrix $A\in \mathbb{R}^{m \times N}$ fulfills the $\ell_1$-NSP of order $s$ with constant $\rho_s < 1/2$. Let the IRLS iterates  $\left\{    x^{k} \right\}_k$ and $ \left\{ \varepsilon_{k} \right\}_k$ be defined by the IRLS algorithm \eqref{step_1} and \eqref{step_2} with initialization $x^0$.
Then, for all $k \in \mathbb{N}$, it holds that
\begin{equation}\label{equ:linearconvergence1}
\mathcal{J}_{\varepsilon_{k}}(x^{k}) -  \onenorm{x_{*}} \le \left(  1-  \frac{c}{\rho_1 N}   \right)^k \left(    \mathcal{J}_{\varepsilon_{0}}(x^0) -  \onenorm{x_{*}} \right)
\end{equation}
as well as
\begin{equation}\label{ineq:globalconvergenerate1}
\onenorm{x^k-x_{*}}\le 9 \left(  1-  \frac{c}{\rho_1 N}   \right)^k  \onenorm{x^0 - x_{*}}.
\end{equation}
Here $c= 1/768$ is an absolute constant and $\rho_1 < 1/2$ denotes the $\ell_1$-NSP constant of order $1$.
\end{theorem}

Inequality \eqref{equ:linearconvergence1} says that the difference $\mathcal{J}_{\varepsilon_{k}}(x^{k}) -  \onenorm{x_{*}}  $ converges linearly with a uniform upper bound of $1-\frac{c}{\rho_1 N}$ on the linear convergence factor. As our proof, which is detailed in the supplementary material, shows, this implies inequality \eqref{ineq:globalconvergenerate1}, which implies that also $\onenorm{ x_{*} -x^k }$ exhibits linear convergence in the number of iterations $k$. In particular, this means that for some error tolerance $ \delta >0 $, we obtain $ \onenorm{ x_{*} -x^k }\le \delta$ after $ O\left(  \rho_1 N  \log \left(  \frac{\onenorm{x_{\ast} -x^0}}{\delta}  \right)    \right) $ iterations.

\begin{remark}
Note that it follows directly from Definition \ref{def:NSP:statement} that the constant $\rho_1$ of the $\ell_1$-NSP of order $1$ satisfies $\rho_1 \le \rho_s \le 1$, which implies that $\delta$-accuracy is obtained after $ O\left(N  \log \left(  \frac{\onenorm{x_{\ast} -x^0}}{\delta}  \right)    \right)$ iterations. This bound can be improved in many scenarios where one can obtain more explicit bounds on $\rho_1$, for example, when $A$ is a Gaussian matrix. Namely, inspecting \cite[p. 142 and Thm. 9.2]{FoucartRauhut13}, we observe in this scenario that $\rho_1 \lesssim \sqrt{(\log N)/m} $ with high probability. Hence, in this scenario, at most  $ O\left(  N \sqrt{\frac{\log N}{m}}   \log \left(  \frac{\onenorm{x_{\ast} -x^0}}{\delta}  \right) \right) $ iterations are needed to achieve $\delta$-accuracy. \end{remark}

The key idea in our proof is to use fact that the quadratic functional $Q_{\varepsilon_k}(\cdot, x^k)$ approximates the $\ell_1$-norm in a neighborhood of the current iterate $x^k $. For this reason, we also expect that for $t>0$ sufficiently small, we have that $Q_{\varepsilon_k}(x^k + t v^k, x^k) < Q_{\varepsilon_k}(x^k, x^k) $ if $v^k = x_{*} - x^k $ is the vector between $x^k$ and the ground truth $x_{*}$. Then by choosing $t$ properly, we can guarantee a sufficient decrease of the functional $\mathcal{J}_{\varepsilon_k} \left(x^k\right) $ in each iteration.

In \Cref{section:numericalexperiments}, we conduct experiments that indeed verify the linear convergence in \eqref{equ:linearconvergence1} and \eqref{ineq:globalconvergenerate1}. Moreover, we study numerically whether one can observe a dependence of the convergence rate on the problem parameters $N$, $s$ and $m$. We construct a worst-case example which indicates that the convergence rate indeed may depend on the dimension $N$ in a way as described by \eqref{equ:linearconvergence1}. In a certain sense, this indicates that there are two convergence phases, a global one, where a dimension-dependent constant cannot be avoided and a local convergence phase, where a local convergence result such as described in Proposition \ref{local_thm_daubechies} kicks in.

Finally, let us mention that we have undertaken no efforts to optimize the constant $c=1/768$ in Theorem \ref{mainresult:sparse}. Nevertheless, we note that the constant $c$ can be replaced by the sharper constant $c_{\rho_s}$ as defined in Proposition \ref{thm:p1:linearrate}.

\section{Numerical experiments}\label{section:numericalexperiments}
In this section, we support our theory with numerical experiments.
First, we examine whether IRLS indeed exhibits two distinct convergence phases, a ``global'' one, as described in this paper, and a local one, as described in \cite{Daubechies10,AravkinBurkeHe19}, corresponding to different linear convergence rate factors. Second, we explore to which extent the dimension dependence in the  convergence rates \eqref{equ:linearconvergence1} and \eqref{ineq:globalconvergenerate1} indicated by \Cref{mainresult:sparse} is necessary, or if we rather can expect a dimension-free linear convergence rate factor. All experiments are conducted on an iMac computer with a 4 GHz Quad-Core Intel Core i7 CPU, using MATLAB R2020b.

\subsection{Local and global convergence phase} \label{section:numerical:localglobal}

We first note that the local convergence result of \cite[Theorem 6.1]{Daubechies10} depends on the locality condition $\xi(k):= \frac{\|x^k - x_{*}\|_1}{\min_{ i \in S} |(x_*)_{i}|} < 1$, cf. \eqref{eq:DDFG:closeness:assumption}. Under this condition (and an appropriate null space condition), \Cref{local_thm_daubechies} stated above implies that
$
\|x^{k+1} - x_{*} \|_1 \leq \mu \|x^{k} - x_{*} \|_1
$
with an absolute constant $\mu < 1$ which, in particular, does \emph{not} depend on the dimension $N, m$, and $s$. This corresponds to a locally linear rate for IRLS. A very similar condition to \eqref{eq:DDFG:closeness:assumption} is required by the comparable and more recent local convergence statement \cite[Theorem III.6, inequality (III.14)]{AravkinBurkeHe19} for the IRLS variant considered in \cite{AravkinBurkeHe19}.

However, a closer look at the locality condition \eqref{eq:DDFG:closeness:assumption} reveal that its \emph{basin of attraction} is very restrictive: This condition means that the \emph{support identification} problem underlying the sparse recovery \emph{has already been solved}, as can be seen from the following proposition, whose proof we provide in the supplementary material.

\begin{proposition} \label{prop:DDFG:supportcondition}
Let $x^k, x_{*} \in \R^N$, let $S \subset [N]$ be the support set of $x_{*}$ of size $|S|= s$. If \eqref{eq:DDFG:closeness:assumption} holds, i.e., if $\|x^k - x_{*}\|_1 < \min_{ i \in S} |(x_*)_{i}|$, then the set $S_k \subset [N]$ of the $s$ largest coordinates of $x^k$ coincides with $S$.
\end{proposition}

    We now explore the behavior of the IRLS algorithm for $\ell_1$-minimization, \Cref{alg:algo1}, and the sharpness of \Cref{prop:DDFG:supportcondition} in experiments that build on those of \cite[Section 8.1]{Daubechies10}. For this purprose, for $N=8000$, we sample independently a $200$-sparse vector $x_{*} \in \R^{N}$ with random support $S \subset [N]$, $s=200 = |S|$, chosen uniformly at random such that $(x_{*})_{S}$ is chosen according the Haar measure on the sphere of a $200$-dimensional unit $\ell_2$-ball, and a measurement matrix $A \in \R^{m \times N}$ with i.i.d. Gaussian entries such that $A_{ij} \sim \mathcal{N}(0,1/m)$, while setting $ m = \lfloor 2 s \log(N / s) \rfloor$. Such a matrix is known to fulfill with high probability the $\ell_1$-null space property of order $s$ with constant $\rho_s < 1$ \cite[Theorem 9.29]{FoucartRauhut13}.

 \begin{figure}[t]
 \centering
      \setlength\figureheight{23mm} 
      \setlength\figurewidth{77mm}
%
%
\definecolor{mycolor1}{rgb}{0.00000,0.44700,0.74100}%
\begin{tikzpicture}

\begin{axis}[%
width=1.0\figurewidth,
height=\figureheight,
scale only axis,
xmin=1,
xmax=64,
xtick={0,4,8,12,16,20,24,28,32,36,40,44,48,52,56,60,64},
xlabel style={font=\color{white!15!black}},
xlabel={iteration $k$},
every outer y axis line/.append style={mycolor1},
every y tick label/.append style={font=\color{mycolor1}},
every y tick/.append style={mycolor1},
ymin=0,
ymax=1,
ylabel style={font=\color{mycolor1}},
ylabel={Estimates of linear rates},
axis x line*=bottom,
axis y line*=left,
legend style={font=\fontsize{7}{30}\selectfont, legend cell align=left, align=left,at={(0.25,0.55)}, draw=white!15!black},
xlabel style={font=\tiny},ylabel style={font=\tiny},
]
\addplot [color=mycolor1, mark size=2.5pt, mark=o, mark options={solid, mycolor1}]
  table[row sep=crcr]{%
2	0.68684067318776\\
3	0.707196351365127\\
4	0.724210700726172\\
5	0.720501752838864\\
6	0.708816872998288\\
7	0.706592152517832\\
8	0.7098469642489\\
9	0.713385506738703\\
10	0.714983117844508\\
11	0.712546374132824\\
12	0.707747452925432\\
13	0.702846774374557\\
14	0.700370852544768\\
15	0.697087208921157\\
16	0.693759379833807\\
17	0.692269948223974\\
};
\addlegendentry{$\mu_{\ell_1}(k)$}

\addplot [color=mycolor1, mark size=2.5pt, mark=*, mark options={solid, fill=blue, mycolor1}]
  table[row sep=crcr]{%
18	0.690014937358868\\
19	0.687461809599895\\
20	0.685921421883023\\
21	0.685430397974599\\
22	0.685356017906471\\
23	0.685428899838826\\
24	0.685595378778444\\
25	0.685817930317256\\
26	0.686028788456298\\
27	0.68622389779091\\
28	0.686405697124002\\
29	0.686590080405637\\
30	0.686750750515174\\
31	0.686881963239259\\
32	0.687002402584028\\
33	0.687109706772119\\
34	0.687217959649388\\
35	0.687311989889039\\
36	0.687391522573151\\
37	0.68745814175324\\
38	0.687525690376276\\
39	0.687580226931345\\
40	0.687625883107827\\
41	0.68766922027423\\
42	0.687693319378354\\
43	0.687714137584507\\
44	0.68764772602303\\
45	0.687479489930666\\
46	0.687243178376491\\
47	0.686447267109216\\
48	0.68476939855563\\
49	0.681581164563298\\
50	0.675695745015432\\
51	0.661902148726568\\
52	0.637840470285411\\
53	0.574715916015535\\
54	0.440822098908626\\
55	0.278187611143494\\
56	0.282355042241448\\
57	0.282845307479207\\
58	0.282846921308376\\
59	0.282848746345579\\
60	0.28288095839649\\
61	0.282995556823607\\
62	0.283242475222693\\
63	0.288471844010758\\
};
\end{axis}
\begin{axis}[width=1.0\figurewidth,
height=\figureheight,
scale only axis,
xmin=1,
xmax=64,
xtick={0,4,8,12,16,20,24,28,32,36,40,44,48,52,56,60,64},
xlabel style={font=\color{white!15!black}},
xlabel={iteration $k$},
every outer y axis line/.append style={red},
every y tick label/.append style={font=\color{red}},
every y tick/.append style={red},
ytick={-10,-8,-6,-4,-2,0,2,4,6},
yticklabels={$10^{-10}$,$10^{-8}$,$10^{-6}$,$10^{-4}$,$10^{-2}$,$10^{0}$,$10^{2}$,$10^{4}$,$10^{6}$},
ylabel style={font=\color{red}},
ymin=-10,
ymax=6,
axis y line*=right,
ylabel={Error parameter $\zeta$},
legend style={font=\fontsize{7}{30}\selectfont, legend cell align=left, align=left, at={(0.22,0.3)}, draw=white!15!black},
xlabel style={font=\tiny},ylabel style={font=\tiny},
]
\addplot [color=red, mark size=2.5pt, mark=x, mark options={solid, red}]
  table[row sep=crcr]{%
1	4.97993685121472\\
2	4.81679285642518\\
3	4.66633286778246\\
4	4.52619780531951\\
5	4.38383284702374\\
6	4.23436689416858\\
7	4.08353570397682\\
8	3.93470043333526\\
9	3.78802471522666\\
10	3.64232050260322\\
11	3.49513363707268\\
12	3.34501195215338\\
13	3.1918726081878\\
14	3.03720067186652\\
15	2.88048778566611\\
16	2.72169665362751\\
17	2.56197213271291\\
18	2.40083062510569\\
19	2.23807920202149\\
20	2.07435356846113\\
21	1.91031692939519\\
22	1.74623315993344\\
23	1.58219557163501\\
24	1.41826345321966\\
25	1.25447228852768\\
26	1.09081462931515\\
27	0.927280467631633\\
28	0.763861347104054\\
29	0.600558871716343\\
30	0.437358014427682\\
31	0.274240126861024\\
32	0.111198382737404\\
33	-0.0517755334170894\\
34	-0.214681032659818\\
35	-0.377527112587175\\
36	-0.540322940797616\\
37	-0.703076680999343\\
38	-0.865787750162211\\
39	-1.02846437117872\\
40	-1.19111215546252\\
41	-1.35373256948637\\
42	-1.51633776409332\\
43	-1.67892981171315\\
44	-1.84156380054929\\
45	-2.00430405447289\\
46	-2.16719361677318\\
47	-2.33058643642853\\
48	-2.49504209237535\\
49	-2.66152451214143\\
50	-2.83177332804276\\
51	-3.01097953710628\\
52	-3.2062674658117\\
53	-3.44681424128631\\
54	-3.80255088322502\\
55	-4.35821309808596\\
56	-4.9074175503974\\
57	-5.45586857238652\\
58	-6.00431711643068\\
59	-6.55276285824873\\
60	-7.10115914353001\\
61	-7.64937952659984\\
62	-8.19722114562016\\
63	-8.73711771491348\\
64	-9.04040296122565\\
};
\addlegendentry{$\zeta(k)$}

\addplot [color=red, line width=0.8pt]
  table[row sep=crcr]{%
1	0\\
64	0\\
};

\addplot [color=red, dashed, line width=1.0pt]
  table[row sep=crcr]{%
33	-10\\
33	6\\
};

\end{axis}
\end{tikzpicture}%
 \caption{Instantaneous linear convergence rates of IRLS for $\ell_1$-minimization ($N = 8000$): Linear convergence factors $\mu_{\ell_1}(k) := \|x^{k}-x_{*}\|_1 /  \|x^{k-1}-x_{*}\|_1$ (in blue), filled blue circle if $S_k = S$ with $S_k$ of \Cref{prop:DDFG:supportcondition} (support identification), and error parameter $\zeta(k) : = \|x^{k}-x_{*}\|_1 / \min_{ i \in S} |(x_*)_{i}|$ (in red). Horizontal (red) line: Threshold $\zeta = 1$. Vertical (red) line: First iterate $k$ with $\zeta(k) < 1$.}
 \label{fig:convrate:standardinit}
 \end{figure}
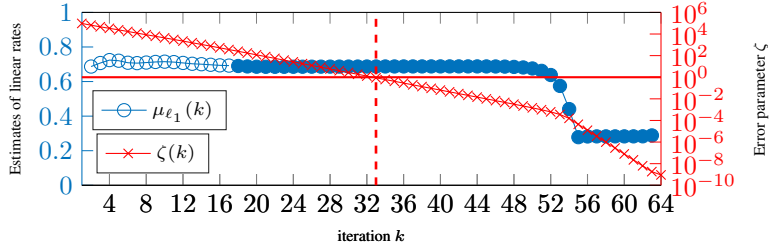

In \Cref{fig:convrate:standardinit}, we track the decay of the $\ell_1$-error $\|x^{k}-x_{*}\|_1$ of the iterates $x^{k}$ returned by \Cref{alg:algo1} via the values of $\zeta(k):=\|x^{k}-x_{*}\|_1 / \min_{ i \in S} |(x_*)_{i}|$, depicted in red, and the behavior of the factor $\mu_{\ell_1}(k):=  \|x^{k}-x_{*}\|_1 /  \|x^{k-1}-x_{*}\|_1$, depicted in blue. We observe that the condition \eqref{eq:DDFG:closeness:assumption} for local convergence with the fast, dimension-less linear rate \eqref{eq:DDFG:localrate} is satisfied after $k=33$ iterations, as indicated by the vertical dashed red line. 

In the first few iterations, $\zeta(k)$ is larger than $1$ by several orders of magnitudes, suggesting that the local convergence rate results of \cite{Daubechies10,AravkinBurkeHe19} do \emph{not} apply until the later stages of the simulation: In fact, we observe that the support $S$ of $x_{*}$ is already perfectly identified via the $s$ largest coordinates of $x^k$ as soon as $k \geq 18$. For iterations $18 \leq k \leq 50$, the linear rate $\mu_{\ell_1}(k)$ remains very stably around $\approx 0.7$, after which an accelerated linear rate can be observed.\footnote{The latter phenomenon cannot be observed for the IRLS algorithm of \cite{Daubechies10} as it uses a slightly different objective function than \Cref{alg:algo1}.} Before $k= 18$, for this example, the rate $\mu(k)$ hovers around $0.7$ with slight variations. For all iterations $k$, $\mu(k)$ is smaller than $1$, in line with the global linear convergence rate implied by \Cref{mainresult:sparse}.

Repeating a similar experiment for a larger ambient space dimension $N=16000$ and a smaller measurement-to-sparsity ratio such that $ m = \lfloor 1.75 s \log(N / s) \rfloor$ results in a qualitatively similar situation, as seen in \Cref{fig:convrate:standardinit:large}(a): In \Cref{fig:convrate:standardinit:large}(a), we add also a plot of the linear convergence factor $\mu(k):=\frac{\mathcal{J}_{\varepsilon_{k}}(x^{k}) -  \onenorm{x_{*}}}{ \mathcal{J}_{\varepsilon_{k-1}}(x^{k-1}) -  \onenorm{x_{*}}}$ that tracks the behavior of the linear convergences in the smoothed $\ell_1$-norm objective $\mathcal{J}$, cf. \eqref{eq:IRLS:suffdecrease}. In addition to what have been observed in \Cref{fig:convrate:standardinit}, we see that $\mu(k)$ and $\mu_{\ell_1}(k)$ exhibit a very similar behavior for this example. 

Hence, these experiments indicate that we can distinguish two phases. In the first, global phase  linear convergence already sets in, but the instantaneous linear convergence rate has not yet stabilized. In the second one, when the support identification problem has been solved, the instantaneous linear convergence stabilizes. 

\begin{remark} There are other methods in the literature, such as proximal algorithms, for which convergence results with a two phase behaviour were already established. For example, \cite{liang2017activity} showed that a forward-backward method applied to the Lasso problem exhibits local linear convergence, and that after a finite number of iterations, the region of fast convergence is reached. In particular, \cite[Proposition 3.6(ii)]{liang2017activity} provides a bound on this number of iterations, which scales proportionally with $||x_*-x^0||_2^2$. On the other hand, \eqref{mainresult:sparse} for IRLS provides a bound on the number of iterations until the fast linear convergence rate is reached that scales proportionally with $\log(||x_*-x^0||_2)$, but also proportionally with the dimension $N$. Moreover, most of these results require stronger assumptions than the NSP, such as the restricted isometry property or a restricted strong convexity/smoothness property.
\end{remark}

 \begin{figure}[h]
 \begin{subfigure}[b]{0.5\textwidth}
     \setlength\figureheight{33mm} 
     \setlength\figurewidth{45mm}
%
%
\definecolor{mycolor1}{rgb}{0.46667,0.67451,0.18824}%
\definecolor{mycolor2}{rgb}{0.00000,0.44700,0.74100}%
\begin{tikzpicture}

\begin{axis}[%
width=1.05\figurewidth,
height=\figureheight,
at={(0\figurewidth,0\figureheight)},
scale only axis,
unbounded coords=jump,
xmin=0,
xmax=101,
xtick={0,10,20,30,40,50,60,70,80,90,100},
xlabel style={font=\color{white!15!black}},
xlabel={iteration $k$},
every outer y axis line/.append style={mycolor1,mycolor2},
every y tick label/.append style={font=\color{mycolor2}},
every y tick/.append style={mycolor1},
ymin=0,
ymax=1,
ylabel style={font=\color{mycolor1}},
ylabel={Estimates of linear rates},
axis background/.style={fill=white},
axis x line*=bottom,
axis y line*=left,
legend style={font=\fontsize{8}{30}\selectfont, legend cell align=left, align=left, at={(0.45,0.5)}, draw=white!15!black},
xlabel style={font=\tiny},ylabel style={font=\tiny},
]
\addplot [color=mycolor2, mark size=2.5pt, mark=o, mark options={solid, mycolor2}]
  table[row sep=crcr]{%
2	0.737994739872931\\
3	0.729242802767835\\
4	0.762327233397151\\
5	0.801728130532322\\
6	0.824072573044909\\
7	0.834290787610712\\
8	0.838906970697082\\
9	0.841230050667104\\
10	0.845071271464089\\
11	0.848349292502206\\
12	0.847022984429376\\
13	0.841737863210088\\
14	0.835944943221669\\
15	0.830378250107056\\
16	0.825792271736375\\
17	0.822938040368303\\
18	0.82022064983958\\
19	0.815333039432517\\
20	0.811304627947717\\
21	0.809197646701936\\
22	0.806947338869101\\
23	0.805696761593037\\
24	0.806651658296079\\
25	0.808778702220582\\
26	0.810278324272705\\
27	0.810864713879848\\
28	0.810777105905196\\
29	0.807999931602385\\
};
\addlegendentry{$\mu_{\ell_1}(k)$}

\addplot [color=mycolor1,line width=0.7pt,  mark size=2.5pt, mark=o, mark options={solid, mycolor1}]
  table[row sep=crcr]{%
3	0.654149504974455\\
4	0.674948655769022\\
5	0.729680017853859\\
6	0.775253409565895\\
7	0.801612289423115\\
8	0.816796923600094\\
9	0.825817143198606\\
10	0.832218560121504\\
11	0.838738847539613\\
12	0.843871845216768\\
13	0.842295723562042\\
14	0.838365599883311\\
15	0.830585590338281\\
16	0.827277653115182\\
17	0.823908022042588\\
18	0.820632295461646\\
19	0.81757275848679\\
20	0.813240726265107\\
21	0.809668061014396\\
22	0.807111280568645\\
23	0.80572283651578\\
24	0.804942811148425\\
25	0.806381767525927\\
26	0.808630495440788\\
27	0.810669188453481\\
28	0.81171542115629\\
29	0.809846558009863\\
};
\addlegendentry{$\mu(k)$}

\addplot [color=mycolor1, mark size=2.5pt, mark=*, mark options={solid, fill=mycolor1}]
  table[row sep=crcr]{%
30	0.807720480803212\\
31	0.80478366823543\\
32	0.803809657429639\\
33	0.803612124569881\\
34	0.803610781665043\\
35	0.803683732076394\\
36	0.80377713275341\\
37	0.803872415155666\\
38	0.803966974571665\\
39	0.804088356530059\\
40	0.804191099771103\\
41	0.804295342145963\\
42	0.804406029956979\\
43	0.804511665538286\\
44	0.804608082745531\\
45	0.804698820034218\\
46	0.804792692237271\\
47	0.804882391898346\\
48	0.80496378928477\\
49	0.805036571536931\\
50	0.805106554934952\\
51	0.805174770704014\\
52	0.805233303173566\\
53	0.805290646162078\\
54	0.805344812766819\\
55	0.805396070803354\\
56	0.805443122424734\\
57	0.805489837106862\\
58	0.805537203873769\\
59	0.805583296928802\\
60	0.805623093274392\\
61	0.805662602667667\\
62	0.805701039144083\\
63	0.805735247136553\\
64	0.805772550713741\\
65	0.805805048226475\\
66	0.805831084208506\\
67	0.805854767677246\\
68	0.805878189392303\\
69	0.805896261487035\\
70	0.805920236691051\\
71	0.805941565648212\\
72	0.805942690342562\\
73	0.805947005273438\\
74	0.805922859660129\\
75	0.805847721043719\\
76	0.80569662787137\\
77	0.805575119571898\\
78	0.805490794747584\\
79	0.805321292717878\\
80	0.804847635086062\\
81	0.804064845966451\\
82	0.803020151185052\\
83	0.801920986595036\\
84	0.80000472908165\\
85	0.796330571058391\\
86	0.792012794233884\\
87	0.782871032613258\\
88	0.768860707959959\\
89	0.746931967159725\\
90	0.711245932683767\\
91	0.609639269057155\\
92	0.442630424906653\\
93	0.303562089208853\\
94	0.295606344023767\\
95	0.29512861562165\\
96	0.291654430957976\\
97	0.355605588107792\\
98	0.210193086633367\\
};

\addplot [color=mycolor2, mark size=1.5pt, mark=*, mark options={solid, fill=blue, mycolor2}]
  table[row sep=crcr]{%
30	0.805502240900026\\
31	0.804204739149256\\
32	0.803860120001559\\
33	0.803795274410831\\
34	0.803833693204481\\
35	0.803900665828422\\
36	0.803974467555093\\
37	0.804050205601129\\
38	0.804162391979801\\
39	0.804252459685642\\
40	0.804345562908054\\
41	0.804448664665102\\
42	0.804548607219374\\
43	0.804637817532068\\
44	0.804721279617368\\
45	0.804814972549011\\
46	0.804905011408366\\
47	0.804985559896537\\
48	0.805055653225509\\
49	0.805124741845759\\
50	0.805192671224802\\
51	0.805248352581174\\
52	0.805304495031123\\
53	0.805357151672482\\
54	0.805407329212727\\
55	0.805452296781141\\
56	0.805497524793068\\
57	0.805544446408717\\
58	0.805590157247721\\
59	0.805627901291061\\
60	0.805666668431357\\
61	0.80570460799774\\
62	0.805737527420447\\
63	0.805775011274719\\
64	0.805806493679924\\
65	0.80583218619546\\
66	0.805854497677728\\
67	0.805877929081504\\
68	0.805893316816054\\
69	0.805919200674651\\
70	0.805945633925532\\
71	0.805937905997666\\
72	0.805950537401529\\
73	0.805924656607232\\
74	0.8058895752675\\
75	0.805689932659291\\
76	0.805623482980522\\
77	0.805486600365001\\
78	0.805437388773342\\
79	0.805014655595926\\
80	0.804256221513356\\
81	0.803303112134302\\
82	0.802131656987086\\
83	0.800746002623995\\
84	0.796619524687173\\
85	0.793425706056013\\
86	0.784383167691843\\
87	0.770967420162293\\
88	0.747619363996634\\
89	0.718110876746366\\
90	0.632385553512833\\
91	0.449368339423657\\
92	0.290740129549903\\
93	0.294704960163997\\
94	0.295668186878743\\
95	0.295667358606337\\
96	0.295672303708258\\
97	0.295676636622758\\
98	0.295822640132243\\
99	0.296041036404276\\
100	0.299295136259204\\
};
\end{axis}

\begin{axis}[%
width=1.05\figurewidth,
height=\figureheight,
at={(0\figurewidth,0\figureheight)},
scale only axis,
xmin=0,
xmax=101,
xlabel style={font=\color{white!15!black}},
xlabel={iteration $k$},
every outer y axis line/.append style={red},
every y tick label/.append style={font=\color{red}},
every y tick/.append style={red},
ymin=-9,
ymax=7,
ytick={-8,-6,-4,-2,0,2,4,6},
yticklabels=\empty,
ylabel style={font=\color{red}},
axis y line*=right,
legend style={font=\fontsize{9}{30}\selectfont, legend cell align=left, align=left, at={(0.4,0.2)}, draw=white!15!black},
xlabel style={font=\tiny},ylabel style={font=\tiny},
]
\addplot [color=red, mark size=2.5pt, mark=x, mark options={solid, red}]
  table[row sep=crcr]{%
1	4.76541972214669\\
2	4.63347298850591\\
3	4.49634514006911\\
4	4.37848657486321\\
5	4.28251369697953\\
6	4.19847915708084\\
7	4.11979660513317\\
8	4.04351040821099\\
9	3.96842518649749\\
10	3.89531852443924\\
11	3.82389322642825\\
12	3.75178842173584\\
13	3.67696528486678\\
14	3.59914295986273\\
15	3.51841892514771\\
16	3.43528973931493\\
17	3.35065687739584\\
18	3.26458757626402\\
19	3.17592261769574\\
20	3.08510657103884\\
21	2.99316118212988\\
22	2.90000637585467\\
23	2.8061779939047\\
24	2.71286402485989\\
25	2.62069373120414\\
26	2.52932795247005\\
27	2.43827635425642\\
28	2.34717783117993\\
29	2.25458915519127\\
30	2.16065590815155\\
31	2.06602253620755\\
32	1.97120301978606\\
33	1.87634846839085\\
34	1.78151467435997\\
35	1.68671706272559\\
36	1.59195931946409\\
37	1.49724248678814\\
38	1.40258624551082\\
39	1.30797864331145\\
40	1.21342131372864\\
41	1.11891964884608\\
42	1.02447193619904\\
43	0.930072376514237\\
44	0.835717862243093\\
45	0.741413909521211\\
46	0.64715854087632\\
47	0.552946630792633\\
48	0.458772534829183\\
49	0.364635707741509\\
50	0.270535521075474\\
51	0.176465366065942\\
52	0.0824254893006291\\
53	-0.0115859910738811\\
54	-0.105570413701451\\
55	-0.199530589440564\\
56	-0.293466379223336\\
57	-0.387376871342697\\
58	-0.481262720002593\\
59	-0.57512822128658\\
60	-0.668972824647339\\
61	-0.762796977172747\\
62	-0.856603385686979\\
63	-0.950389590782764\\
64	-1.04415882790664\\
65	-1.13791421810828\\
66	-1.23165758394598\\
67	-1.32538832221698\\
68	-1.41911076798589\\
69	-1.51281926521322\\
70	-1.60651351829954\\
71	-1.70021193570213\\
72	-1.79390354649349\\
73	-1.88760910363279\\
74	-1.98133356572027\\
75	-2.07516562868393\\
76	-2.16903351177862\\
77	-2.26297519164901\\
78	-2.35694340576044\\
79	-2.4511396188255\\
80	-2.54574518965128\\
81	-2.64086574015316\\
82	-2.73662008357682\\
83	-2.83312530426179\\
84	-2.93187435774963\\
85	-3.03236809075515\\
86	-3.13783982531191\\
87	-3.25080379945641\\
88	-3.3771232579963\\
89	-3.52093175325053\\
90	-3.7199498130867\\
91	-4.06734734246236\\
92	-4.60384236284441\\
93	-5.13445491734493\\
94	-5.66365031925196\\
95	-6.19284693777497\\
96	-6.72203629268753\\
97	-7.25121928330067\\
98	-7.78018787459908\\
99	-8.30883595864601\\
100	-8.83273629931651\\
101	-8.98958758188758\\
};
\addlegendentry{$\zeta(k)$}

\addplot [color=red, line width=0.8pt]
  table[row sep=crcr]{%
1	0\\
101	0\\
};

\addplot [color=red, dashed, line width=1.0pt]
  table[row sep=crcr]{%
53	-9\\
53	7\\
};

\end{axis}
\end{tikzpicture}%
 \captionsetup{justification=centering,margin=1mm}
 \caption{Standard initialization (uniform \\ weights $(w_0)_i = 1$ for all $i$).}
 \end{subfigure}
 \begin{subfigure}[b]{0.5\textwidth}
     \setlength\figureheight{33mm} 
     \setlength\figurewidth{45mm}\hspace*{3mm}
%
%
\definecolor{mycolor1}{rgb}{0.46667,0.67451,0.18824}%
\definecolor{mycolor2}{rgb}{0.00000,0.44700,0.74100}%
\begin{tikzpicture}

\begin{axis}[%
width=1.0\figurewidth,
height=\figureheight,
at={(0\figurewidth,0\figureheight)},
scale only axis,
unbounded coords=jump,
xmin=0,
xmax=100,
xtick={0,10,20,30,40,50,60,70,80,90,100},
xlabel style={font=\color{white!15!black}},
xlabel={iteration $k$},
every outer y axis line/.append style={mycolor2},
every y tick label/.append style={font=\color{mycolor2}},
every y tick/.append style={mycolor2},
ymin=0,
ymax=1,
axis background/.style={fill=white},
axis x line*=bottom,
axis y line*=left,
legend style={font=\fontsize{8}{30}\selectfont, legend cell align=left, align=left,  at={(0.45,0.5)}, draw=white!15!black},
xlabel style={font=\tiny},ylabel style={font=\tiny},
]
\addplot [color=blue, mark size=2.5pt, mark=o, mark options={solid, blue}]
  table[row sep=crcr]{%
2	0.989937038967682\\
3	0.968546552936381\\
4	0.945739409718627\\
5	0.922799593278974\\
6	0.903450848797426\\
7	0.888256095302694\\
8	0.873938262469174\\
9	0.862076232459906\\
10	0.851888075617552\\
11	0.843690808881313\\
12	0.838214374354926\\
13	0.834805923745606\\
14	0.832711696396053\\
15	0.831190065544973\\
16	0.828912894592705\\
17	0.823859030418603\\
18	0.814992952305936\\
19	0.805924480957863\\
20	0.799850558314915\\
21	0.795571688985866\\
22	0.792386335595644\\
23	0.78956230291093\\
24	0.787328741063433\\
25	0.785710221595671\\
26	0.784176335120897\\
27	0.782847395569583\\
28	0.782612572907667\\
29	0.784087517725716\\
30	0.786184211820327\\
31	0.786393606655757\\
32	0.78532404294274\\
33	0.785004236811908\\
34	0.784811657891758\\
35	0.784967296981115\\
36	0.785773503479\\
37	0.786950715088217\\
38	0.787259703208939\\
};
\addlegendentry{$\mu_{\ell_1}(k)$}

\addplot [color=mycolor1, line width=0.7pt, mark size=2.5pt, mark=o, mark options={solid, mycolor1}]
  table[row sep=crcr]{%
2	0.979673746942339\\
3	0.964429374021478\\
4	0.944777792446196\\
5	0.928215612148456\\
6	0.918012859382131\\
7	0.902376710271128\\
8	0.869510610740377\\
9	0.856239036646947\\
10	0.845144060424161\\
11	0.836141614944118\\
12	0.830490390252301\\
13	0.828128067512857\\
14	0.82852286442008\\
15	0.828586511578777\\
16	0.828320322266265\\
17	0.827215082731811\\
18	0.821519128706093\\
19	0.811782448903507\\
20	0.802962692203778\\
21	0.797747836560021\\
22	0.794355510245175\\
23	0.790217400128832\\
24	0.788196832594558\\
25	0.786307903890226\\
26	0.784068971762096\\
27	0.783239177018551\\
28	0.782122477267812\\
29	0.78216264385145\\
30	0.783955488666891\\
31	0.785597485849863\\
32	0.785916921667457\\
33	0.784651569538987\\
34	0.784376818643693\\
35	0.784755619981453\\
36	0.784827380190912\\
37	0.785820508766029\\
38	0.786995743720475\\
};
\addlegendentry{$\mu(k)$}

\addplot [color=mycolor1, mark size=2.5pt, mark=*, mark options={solid, fill=mycolor1}]
  table[row sep=crcr]{%
39	0.787215300914553\\
40	0.786585288908427\\
41	0.785421731581508\\
42	0.78529650466335\\
43	0.785454699256945\\
44	0.785659814252469\\
45	0.785862612975095\\
46	0.786050179809607\\
47	0.786217441100504\\
48	0.786383762559737\\
49	0.786543375529163\\
50	0.786687215070458\\
51	0.786822835344267\\
52	0.786953606584149\\
53	0.787065019256531\\
54	0.787165142865661\\
55	0.78726184313696\\
56	0.787350716355348\\
57	0.787430602328769\\
58	0.787515359411572\\
59	0.787595316824024\\
60	0.787670277706951\\
61	0.787740202459744\\
62	0.787801397321886\\
63	0.787857731041157\\
64	0.787909086635977\\
65	0.787951291803026\\
66	0.788001617373411\\
67	0.788036855265503\\
68	0.788081754583814\\
69	0.788111616709287\\
70	0.788142210341661\\
71	0.788162700843756\\
72	0.788150810653491\\
73	0.788102461839628\\
74	0.78804509884614\\
75	0.787856651075685\\
76	0.787600561963018\\
77	0.787201590578832\\
78	0.786790634744605\\
79	0.78579031693134\\
80	0.784469782122528\\
81	0.783076154036692\\
82	0.779649811365516\\
83	0.774288619819956\\
84	0.766838052419587\\
85	0.752720100207995\\
86	0.729574956595584\\
87	0.68160816888637\\
88	0.576905999342145\\
89	0.421723509523844\\
90	0.292988306432619\\
91	0.291523795420545\\
92	0.290521573091567\\
93	0.29024977357156\\
94	0.245248385535731\\
95	0.0190515071102908\\
};

\addplot [color=mycolor2, mark size=1.5pt, mark=*, mark options={solid, fill=blue, mycolor2}]
  table[row sep=crcr]{%
39	0.786383766782677\\
40	0.785759146974016\\
41	0.785479082909958\\
42	0.785580527561137\\
43	0.785756826173525\\
44	0.785941612544917\\
45	0.786113744207299\\
46	0.786265642765187\\
47	0.786427415115682\\
48	0.786584114393776\\
49	0.78672260631852\\
50	0.786856215589694\\
51	0.786988616181873\\
52	0.787098095263616\\
53	0.787193819756425\\
54	0.787287055614357\\
55	0.787370783355996\\
56	0.78744466556951\\
57	0.787527907434097\\
58	0.787606903416389\\
59	0.787681480937735\\
60	0.787751335097781\\
61	0.787810689579516\\
62	0.78786621911226\\
63	0.787916832969515\\
64	0.787954747065953\\
65	0.788006067986231\\
66	0.788039241162096\\
67	0.788085630710995\\
68	0.78811348280648\\
69	0.788142600222855\\
70	0.788171153841873\\
71	0.788155041346101\\
72	0.78809897000724\\
73	0.788076678624904\\
74	0.787905702041243\\
75	0.787690857972635\\
76	0.787248013591338\\
77	0.786916546778372\\
78	0.786206766040323\\
79	0.784601432370081\\
80	0.783741115423538\\
81	0.780677892328129\\
82	0.775205830764662\\
83	0.76888073607347\\
84	0.755260657964371\\
85	0.732918581722473\\
86	0.690992859475983\\
87	0.600208769567132\\
88	0.429099902649867\\
89	0.28100558569264\\
90	0.291525058926581\\
91	0.291571380882615\\
92	0.291572633343378\\
93	0.291584834477803\\
94	0.291611439941818\\
95	0.291605137082755\\
96	0.292292199504985\\
97	0.29540645867705\\
};
\end{axis}

\begin{axis}[%
width=1.0\figurewidth,
height=\figureheight,
at={(0\figurewidth,0\figureheight)},
scale only axis,
xmin=0,
xmax=100,
xlabel style={font=\color{white!15!black}},
xlabel={iteration $k$},
every outer y axis line/.append style={red},
every y tick label/.append style={font=\color{red}},
every y tick/.append style={red},
ymin=-9,
ymax=7,
ytick={-8,-6,-4,-2,0,2,4,6},
yticklabels=\empty,
yticklabels={$10^{-8}$,$10^{-6}$,$10^{-4}$,$10^{-2}$,$10^{0}$,$10^{2}$,$10^{4}$,$10^{6}$},
ylabel style={font=\color{red}},
axis y line*=right,
ylabel={Error parameter $\zeta$},
legend style={font=\fontsize{9}{30}\selectfont, legend cell align=left, align=left, at={(0.4,0.2)}, draw=white!15!black},
xlabel style={font=\tiny},ylabel style={font=\tiny},
]
\addplot [color=red, mark size=2.5pt, mark=x, mark options={solid, red}]
  table[row sep=crcr]{%
1	5.53128311857644\\
2	5.5268906924685\\
3	5.51301119227583\\
4	5.48878267908807\\
5	5.45389007352324\\
6	5.40979460377193\\
7	5.35833280011595\\
8	5.29981355401352\\
9	5.23535922572317\\
10	5.1657417649327\\
11	5.0919250828868\\
12	5.01528018706954\\
13	4.93686570920817\\
14	4.85736037407613\\
15	4.77706071793134\\
16	4.69556961351744\\
17	4.61142251993682\\
18	4.52257637310885\\
19	4.42887072131608\\
20	4.33187957360573\\
21	4.23255889363884\\
22	4.13149587132492\\
23	4.02888227640115\\
24	3.92503838184727\\
25	3.8203007849269\\
26	3.71471451694683\\
27	3.6083916280197\\
28	3.5019384487126\\
29	3.39630298887672\\
30	3.29182730693093\\
31	3.18746728097498\\
32	3.08251617469309\\
33	2.97738817541167\\
34	2.87215362100739\\
35	2.76700518471253\\
36	2.66230256489349\\
37	2.55825009924053\\
38	2.45436812139018\\
39	2.35000266122256\\
40	2.24529210653701\\
41	2.14042673092544\\
42	2.03561744083982\\
43	1.93090560343\\
44	1.82629588701212\\
45	1.72178127644546\\
46	1.6173505752791\\
47	1.51300922001408\\
48	1.40875439130115\\
49	1.30457602101564\\
50	1.20047140079404\\
51	1.09643985111352\\
52	0.992468712540997\\
53	0.888550388340019\\
54	0.784683499274646\\
55	0.680862794841022\\
56	0.57708284011897\\
57	0.473348792840778\\
58	0.369658306937767\\
59	0.266008941894074\\
60	0.162398089792293\\
61	0.0588199591238863\\
62	-0.0447275609929177\\
63	-0.148247182194458\\
64	-0.251745905902726\\
65	-0.355216344145673\\
66	-0.45866850001165\\
67	-0.562095090993185\\
68	-0.665506333645597\\
69	-0.768901531274113\\
70	-0.87228099513243\\
71	-0.975669337315582\\
72	-1.07908857740339\\
73	-1.18252010168597\\
74	-1.28604585822398\\
75	-1.38969005319528\\
76	-1.49357847971065\\
77	-1.59764980223774\\
78	-1.70211302522971\\
79	-1.80746392858652\\
80	-1.9132912979304\\
81	-2.02081941698488\\
82	-2.13140238635368\\
83	-2.2455434063469\\
84	-2.36744654375218\\
85	-2.50239081122918\\
86	-2.66291725170766\\
87	-2.884614915156\\
88	-3.25205649915194\\
89	-3.80334154646615\\
90	-4.33866565465831\\
91	-4.8739207609888\\
92	-5.40917400178768\\
93	-5.94440906950072\\
94	-6.47960451211079\\
95	-7.01480934161767\\
96	-7.54899211626072\\
97	-8.07857212989841\\
98	-8.13012704103876\\
};
\addlegendentry{$\zeta(k)$}

\addplot [color=red, line width=0.8pt]
  table[row sep=crcr]{%
1	0\\
98	0\\
};

\addplot [color=red, dashed, line width=1.0pt]
  table[row sep=crcr]{%
62	-9\\
62	7\\
};

\end{axis}
\end{tikzpicture}%
 \captionsetup{justification=centering,margin=1mm}
 \caption{Adversary initialization \\ (weights $(w_0)_i$  as in \eqref{eq:adv:init}).}
 \end{subfigure}
 \caption{Instantaneous linear convergence rates of IRLS for $\ell_1$-minimization ($N = 16000$): Linear convergence factors $\mu_{\ell_1}(k) := \frac{\|x^{k}-x_{*}\|_1}{\|x^{k-1}-x_{*}\|_1}$ (in blue) and $\mu(k):= \frac{\mathcal{J}_{\varepsilon_{k}}(x^{k}) -  \onenorm{x_{*}}}{ \mathcal{J}_{\varepsilon_{k-1}}(x^{k-1}) -  \onenorm{x_{*}}}$ (in green), 
  filled circles if $S_k = S$ (perfect support identification), and error parameter $\zeta(k) : = \|x^{k}-x_{*}\|_1 / \min_{ i \in S} |(x_*)_{i}|$ (in red), horizontal and vertical red lines as in \Cref{fig:convrate:standardinit}.}
 \label{fig:convrate:standardinit:large}
 \end{figure}
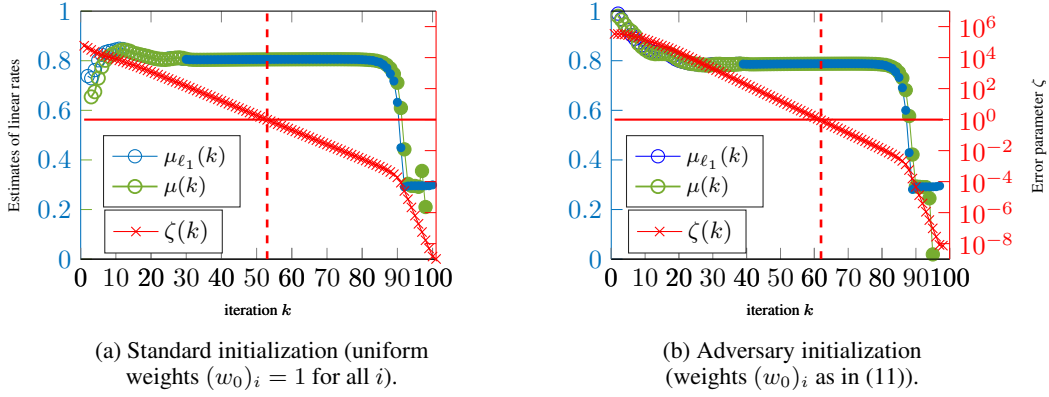

\vspace*{-1mm}

\subsection{Global convergence rate and its dimension dependence} \label{sec:numerics:2}
In this section, we explore to which extent  the dependence on $N$ in the convergence rates \eqref{ineq:globalconvergenerate1} and \eqref{ineq:approxsparse2} is necessary or if we can rather expect a dimension-free linear convergence rate factor. 
To this end, we run a variation of IRLS that initializes the weight vector $w_0 \in \R^N$ not uniformly as in \Cref{alg:algo1}, but based on an \emph{adversary initialization}, here denoted by $z^{\text{adv}}$. More specifically, we first compute a minimizer
\begin{equation*}
z^{\text{adv}}  \in \argmin_{z \in \R^{S^c}: A_{S^c} z   = y} \|z\|_1
\end{equation*}
of the $\ell_1$-minimization problem restricted to the off-support coordinates of $x_{*}$ indexed by $S^c$ and set then $x^0 \in \R^N$ such that $x_{S^c}^0 :=z^{\text{adv}}$ and $x_{S}^0 = 0$. Based on this \emph{initialization} $x^0$, we compute $\varepsilon_0:= \frac{ \sigma_s(x^{0})_{\ell_1}}{N}$ and set the first weight vector such that for all $i \in [N]$,
\begin{equation} \label{eq:adv:init}
\left( w_{0} \right)_i :=  \frac{1}{ \max \left( \vert x_i^{0}  \vert , \varepsilon_{0}\right)},
\end{equation}
before proceeding with the IRLS steps \eqref{step_1}, \eqref{step_2} and \eqref{eq:IRLS:step3} until convergence.

We observe in \Cref{fig:convrate:standardinit:large}(b) that this initialization, which is \emph{adversary} as it sets very large initial weights on the coordinates of $S$ that correspond to the true support of $x_{*}$, eventually results in the same behavior of \Cref{alg:algo1} as for the standard initialization by uniform weights, identifying the true support at iteration $k=39$ compared to $k=30$. However, in the first few iterations, we see that the instantaneous linear convergence factor $\mu(k)$ is close to $1$ with $\mu(1)=0.980$, decreasing only slowly before stabilizing around $0.79$ after around $k=30$. 

While this is just one example, this already indicates that in general, a linear rate such as \eqref{eq:DDFG:localrate}, i.e.,  without dependence on the dimension $N$ (which has been proven locally in \cite[Theorem 6.1]{Daubechies10} and \cite[Theorem III.6]{AravkinBurkeHe19}) might not hold in general.

\begin{figure}[h]
\centering
     \setlength\figureheight{27.5mm} 
     \setlength\figurewidth{120mm}
%
%
\begin{tikzpicture}

\begin{axis}[%
width=0.8\figurewidth,
height=\figureheight,
at={(0\figurewidth,0\figureheight)},
scale only axis,
xmin=0,
xmax=1.6,
xtick={0.05,0.2,0.4,0.8,1.6}, 
xticklabels={500,2000,4000,8000,16000},
xticklabel style={anchor=north, inner sep =4pt},
xlabel style={font=\color{white!15!black}},
xlabel={Dimension $N$},
ymin=0,
ymax=160,
axis background/.style={fill=white},
axis x line*=bottom,
axis y line*=left,
legend style={legend cell align=left, align=left, at={(0.3,0.95)}, draw=white!15!black},
xlabel style={font=\tiny},ylabel style={font=\tiny},
]
\addplot [color=blue, mark size=4.0pt, mark=o, mark options={solid, blue}]
  table[row sep=crcr]{%
0.0125	1.64580517927425\\
0.0176776695296638	1.83601105561502\\
0.0250	2.47873440328476\\
0.0353553390593273	3.44032788844197\\
0.0500	4.90708598998026\\
0.0707106781186547	6.88485259526169\\
0.1000	9.99811291374135\\
0.14142135623731	14.4847212588611\\
0.2000	22.4094960516231\\
0.282842712474619	31.0423630117493\\
0.4000	43.9518811091548\\
0.565685424949238	62.290510457482\\
0.8000	84.0157447257843\\
1.13137084989848	117.390922393224\\
1.6000	155.744340969093\\
};
\addlegendentry{$\frac{1}{1-\mu(1)}$}

\addplot [color=red, mark size=4.0pt, mark=x, mark options={solid, red}]
  table[row sep=crcr]{%
0.0125	1.25\\
0.0176776695296638	1.76776695296576\\
0.0250	2.5\\
0.0353553390593273	3.53553390593333\\
0.0500	5\\
0.0707106781186547	7.07106781186485\\
0.1000	10\\
0.14142135623731	14.1421356237315\\
0.2000	20\\
0.282842712474619	28.2842712474612\\
0.4000	40\\
0.565685424949238	56.5685424949243\\
0.8000	80\\
1.13137084989848	113.137084989847\\
1.6000	160\\
};
\addlegendentry{$\frac{N}{100}$}

\end{axis}
\end{tikzpicture}
 \captionsetup{justification=centering,margin=1mm}
 \caption{Comparison of $\frac{N}{100}$ and $\frac{1}{1-\mu(1)}$ (for which \Cref{thm:p1:linearrate} provides un upper bound of $\frac{\rho_1 N}{c}$) for different dimension parameters $N$, where $\mu(1) = \frac{\mathcal{J}_{\varepsilon_{1}}(x^{1}) -  \onenorm{x_{*}}}{ \mathcal{J}_{\varepsilon_{0}}(x^{0}) -  \onenorm{x_{*}}}$ is the linear convergence factor, for IRLS initialized from adversary initialization.}
 \label{fig:convrate:dimdep}
 \end{figure}
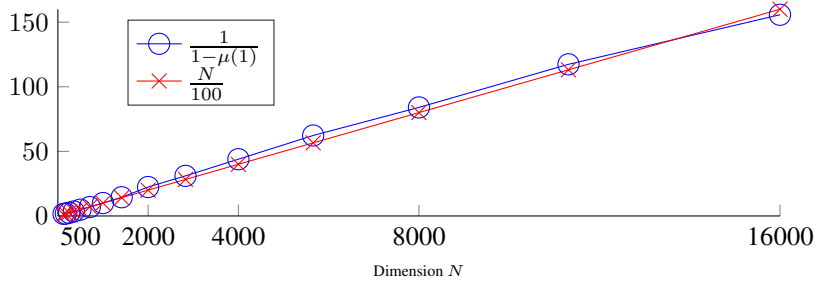

In our next experiment, we further investigate numerically the dimension dependence of the worst-case linear convergence factor $\mu(k):=\frac{\mathcal{J}_{\varepsilon_{k}}(x^{k}) -  \onenorm{x_{*}}}{ \mathcal{J}_{\varepsilon_{k-1}}(x^{k-1}) -  \onenorm{x_{*}}}$, which is upper bounded by the result of \Cref{mainresult:sparse}. We saw that in the experiment using the adversary initialization mentioned above and depicted in \Cref{fig:convrate:standardinit:large}(b), the maximal value was attained in the first iteration, i.e., for $\mu(1)$, as the effect of the adversary initialization is most eminent for $k=1$. 

We now run IRLS starting from the adversary initialization for different ambient dimensions $N= 125\cdot 2^{\ell/2}$ for $\ell=0,1,\ldots, 14$. For each of the values of $N$, we sample vectors $x_{*} \in \R^{N}$ of sparsity $s=40$ from the same random model as above, and scale the number of i.i.d. Gaussian measurements with $ m = \lfloor 2 s \log(N / s) \rfloor$. We average the resulting values for $\mu(1)$ across $500$ independent realizations of the experiment.

In \Cref{fig:convrate:dimdep}, we see that the dependence on $N$ of linear convergence factor $\mu(1)$ that is observed for this experiment is quite well described by the upper bound \eqref{equ:linearconvergence1} provided by our main result \Cref{mainresult:sparse}, as $\frac{1}{1-\mu(1)}$ scales almost linearly with $N$. As a footnote in \Cref{sec:main:results} indicates, the constant $\rho_1$ of the null space property of order $1$ scales with  $  \sqrt{\frac{\log N}{m}}$, and therefore a precise dependence on all parameters, including $m$ and $s$, might be more complicated than what can be observed in this experiment. 

Nevertheless, we interpret \Cref{fig:convrate:dimdep} as strong evidence that the linear convergence rate factor of \Cref{thm:p1:linearrate} is tight in its dependence on $N$, and that a dimension-less factor $\mu$ cannot be expected in general.
\vspace{-4mm}
\section*{Conclusion}
\vspace*{-2mm}
In this paper, we solved an open problem in the algorithmic theory for sparse recovery. In particular, we established a new variant of the IRLS algorithm for Basis Pursuit or $\ell_1$-minimization for which we show a global linear convergence rate under a suitable and sharp assumption, namely, the null-space property. Moreover, we have corroborated our theory with numerical experiments that, first, discussed the difference between the local and global convergence phase and, second, that elucidated the optimality of the dimension dependence of  convergence rate given by our main theorem. 

We think that the results in this paper give rise to a number of interesting research directions for follow-up work. While the numerical experiments in \Cref{section:numericalexperiments} substantiate the hypothesis that the dependence of the convergence rate on $N$ and $\rho_1$ in our theory is not an artifact of our proof, we also observed in this section that for a \textit{generic initialization} no such dependence can be observed. In view of this, it is interesting to investigate whether a dimension-independent global convergence rate is possible, for example via a \textit{smoothed analysis} \cite{spielman_teng, dadush_huiberts}. Furthermore, there are currently no convergence rates available for IRLS optimizing a nuclear norm-type objective, which is of great interest for low-rank matrix recovery \cite{mohan_fazel,fornasier_rauhut_ward,kummerle_sigl}, and we expect that our analysis may be generalizable to this setting as well.

Finally, it was observed that sparse vectors can be recovered from even fewer measurements via the optimization of a non-convex $\ell_p$-quasinorm (with $0 < p < 1$), and that IRLS exhibits excellent performance in this case \cite{chartrand_yin,Daubechies10}. While a thorough understanding has remained elusive so far for this non-convex case, we consider our results as a first step towards a global convergence theory for IRLS for the optimization of $\ell_p$-quasinorms or similar non-convex surrogate objectives.

\vspace*{-4mm}
\begin{ack}
\vspace*{-2mm}
The authors want to thank Massimo Fornasier for inspiring discussions and the anonymous reviewers
for their detailed comments.
C.K. is grateful for support by the NSF under the grant NSF-IIS-1837991. C.M.V. is supported by the German Federal Ministry of Education and Research (BMBF) in the context of the Collaborative Research Project SparseMRI3D+ and by the German Science Foundation (DFG) in the context of the project KR 4512/1-1. D.S. was supported by the Air Force Office of Scientific Research under award \#FA9550-18-1-00788.
\end{ack}

\bibliography{GlobalConv_IRLS.bib}

\begin{bibdiv}
\begin{biblist}

\bib{adil_peng_sachdeva}{inproceedings}{
      author={Adil, Deeksha},
      author={Peng, Richard},
      author={Sachdeva, Sushant},
       title={Fast, provably convergent irls algorithm for p-norm linear
  regression},
        date={2019},
   booktitle={{Advances in Neural Information Processing Systems (NeurIPS)}},
      volume={32},
       pages={14189\ndash 14200},
}

\bib{allain_idier_goussard}{article}{
      author={Allain, Marc},
      author={Idier, Jerome},
      author={Goussard, Yves},
       title={On global and local convergence of half-quadratic algorithms},
        date={2006},
     journal={IEEE Trans. Image Process.},
      volume={15},
      number={5},
       pages={1130\ndash 1142},
}

\bib{AravkinBurkeHe19}{article}{
      author={Aravkin, Aleksandr},
      author={Burke, James~V},
      author={He, Daiwei},
       title={{IRLS for Sparse Recovery Revisited: Examples of Failure and a
  Remedy}},
        date={2019},
     journal={arXiv preprint arXiv:1910.07095},
}

\bib{attouch_bolte_svaiter}{article}{
      author={Attouch, Hedy},
      author={Bolte, Jerome},
      author={Svaiter, Benar~Fux},
       title={Convergence of descent methods for semi-algebraic and tame
  problems: Proximal algorithms, forward-backward splitting, and regularized
  gauss-seidel methods},
        date={2013},
     journal={Math. Programming},
      volume={137},
       pages={91\ndash 129},
}

\bib{ba2013convergence}{article}{
      author={Ba, Demba},
      author={Babadi, Behtash},
      author={Purdon, Patrick~L},
      author={Brown, Emery~N.},
       title={Convergence and stability of iteratively re-weighted least
  squares algorithms},
        date={2013},
     journal={IEEE Transactions on Signal Processing},
      volume={62},
      number={1},
       pages={183\ndash 195},
}

\bib{baraniuk_davenport_devore_wakin}{article}{
      author={Baraniuk, Richard},
      author={Davenport, Mark},
      author={DeVore, Ronald},
      author={Wakin, Michael},
       title={A simple proof of the restricted isometry property for random
  matrices},
        date={2008},
     journal={Constr. Approx.},
      volume={28},
       pages={253\ndash 263},
}

\bib{back_sabach}{article}{
      author={Beck, Amir},
      author={Sabach, Shoham},
       title={Weiszfeld's method: Old and new results},
        date={2015},
     journal={J. Optim. Theory Appl.},
      volume={164},
       pages={1\ndash 40},
}

\bib{Bolte2007lojasiewicz}{article}{
      author={Bolte, J{\'e}r{\^o}me},
      author={Daniilidis, Aris},
      author={Lewis, Adrian},
       title={The {\l}ojasiewicz inequality for nonsmooth subanalytic functions
  with applications to subgradient dynamical systems},
        date={2007},
     journal={SIAM J. Optim.},
      volume={17},
      number={4},
       pages={1205\ndash 1223},
}

\bib{boyd_admm}{article}{
      author={Boyd, Stephen},
      author={Parikh, Neal},
      author={Chu, Eric},
      author={Peleato, Borja},
      author={Eckstein, Jonathan},
       title={Distributed optimization and statistical learning via the
  alternating direction method of multipliers},
        date={2011},
     journal={Foundations and Trends{\textregistered} in Machine Learning},
      volume={3},
       pages={1\ndash 122},
}

\bib{burrus}{article}{
      author={Burrus, Charles~S.},
       title={{Iterative reweighted least squares}},
        date={2012},
     journal={OpenStax CNX: Available online:
  http://cnx.org/contents/92b90377-2b34-49e4-b26f-7fe572db78a1},
}

\bib{cahil_chen_wang}{article}{
      author={Cahill, Jameson},
      author={Chen, Xuemei},
      author={Wang, Rongrong},
       title={The gap between the null space property and the restricted
  isometry property},
        date={2016},
     journal={Linear Algebra Appl.},
      volume={501},
       pages={363\ndash 375},
}

\bib{CRT06b}{article}{
      author={Cand\`{e}s, Emmanuel~J.},
      author={Romberg, Justin~K.},
      author={Tao, Terence},
       title={Stable signal recovery from incomplete and inaccurate
  measurements},
        date={2006},
     journal={Comm. Pure Appl. Math.},
      volume={59},
      number={8},
       pages={1207\ndash 1223},
}

\bib{CT06}{article}{
      author={{Cand{\`e}s}, Emmanuel~J.},
      author={{Tao}, Terence},
       title={{Near-Optimal Signal Recovery From Random Projections: Universal
  Encoding Strategies?}},
        date={2006},
        ISSN={0018-9448},
     journal={IEEE Trans. Inf. Theory},
      volume={52},
      number={12},
       pages={5406\ndash 5425},
}

\bib{candes2008enhancing}{article}{
      author={Cand\`es, Emmanuel~J.},
      author={Wakin, Michael~B.},
      author={Boyd, Stephen~P.},
       title={Enhancing sparsity by reweighted $\ell_1$ minimization},
        date={2008},
     journal={Journal of Fourier analysis and applications},
      volume={14},
      number={5-6},
       pages={877\ndash 905},
}

\bib{chambolle2011first}{article}{
      author={Chambolle, Antonin},
      author={Pock, Thomas},
       title={A first-order primal-dual algorithm for convex problems with
  applications to imaging},
        date={2011},
     journal={Journal of mathematical imaging and vision},
      volume={40},
      number={1},
       pages={120\ndash 145},
}

\bib{chartrand_yin}{inproceedings}{
      author={Chartrand, Rick},
      author={Yin, Wotao},
       title={Iteratively reweighted algorithms for compressive sensing},
        date={2008},
   booktitle={{IEEE International Conference on Acoustics, Speech and Signal
  Processing (ICASSP)}},
       pages={3869\ndash 3872},
}

\bib{Chartrand2016}{incollection}{
      author={Chartrand, Rick},
      author={Yin, Wotao},
       title={Nonconvex sparse regularization and splitting algorithms},
        date={2016},
   booktitle={Splitting methods in communication, imaging, science, and
  engineering},
   publisher={Springer International Publishing},
       pages={237\ndash 249},
}

\bib{ChenDonohoSaunders01}{article}{
      author={Chen, Scott~Shaobing},
      author={Donoho, David~L},
      author={Saunders, Michael~A},
       title={Atomic decomposition by basis pursuit},
        date={2001},
     journal={SIAM Rev.},
      volume={43},
      number={1},
       pages={129\ndash 159},
}

\bib{chen_donoho}{inproceedings}{
      author={Chen, Shaobing~Scott},
      author={Donoho, David~L.},
       title={Basis pursuit},
        date={1994},
   booktitle={{Proc. of 1994 28th Asilomar Conference on Signals, Systems and
  Computers}},
       pages={41\ndash 44},
}

\bib{claerbout_muir}{article}{
      author={Claerbout, Jon~F.},
      author={Muir, Francis},
       title={Robust modeling with erratic data},
        date={1973},
     journal={Geophysics},
      volume={38},
      number={5},
       pages={826\ndash 844},
}

\bib{cohen_dahmen_devore}{article}{
      author={Cohen, Albert},
      author={Dahmen, Wolfgang},
      author={DeVore, Ronald},
       title={Compressed sensing and best k-term approximation},
        date={2009},
     journal={J. Amer. Math. Soc.},
      volume={22},
       pages={211\ndash 231},
}

\bib{dadush_huiberts}{article}{
      author={Dadush, Daniel},
      author={Huiberts, Sophie},
       title={A friendly smoothed analysis of the simplex method},
        date={2018},
     journal={SIAM J. Comput.},
      volume={49},
      number={5},
       pages={STOC18\ndash 449\ndash STOC18\ndash 499},
}

\bib{Daubechies10}{article}{
      author={Daubechies, Ingrid},
      author={DeVore, Ronald},
      author={Fornasier, Massimo},
      author={G\"unt\"urk, Sinan},
       title={Iteratively reweighted least squares minimization for sparse
  recovery},
        date={2010},
     journal={Comm. Pure Appl. Math.},
      volume={63},
       pages={1\ndash 38},
}

\bib{davis_mallat_avellaneda}{article}{
      author={Davis, Geoffrey},
      author={Mallat, Stephane},
      author={Avellanedas, Marco},
       title={Adaptive greedy approximations},
        date={1997},
     journal={Constr. Approx.},
      volume={13},
      number={1},
       pages={57\ndash 98},
}

\bib{diening_fornasier_tomasi_wank}{article}{
      author={Diening, Lars},
      author={Fornasier, Massimo},
      author={Roland, Tomasi},
      author={Wank, Maximilian},
       title={{A Relaxed Ka\u{c}anov iteration for the p-poisson problem}},
        date={2020},
     journal={Numer. Math.},
      volume={145},
      number={2},
       pages={1\ndash 34},
}

\bib{donoho}{article}{
      author={Donoho, David~L.},
       title={Compressed sensing},
        date={2006},
     journal={IEEE Trans. Inf. Theor.},
      volume={52},
      number={4},
       pages={1289\ndash 1306},
}

\bib{DonohoLogan92}{article}{
      author={Donoho, David~L},
      author={Logan, Benjamin~F},
       title={Signal recovery and the large sieve},
        date={1992},
     journal={SIAM J. Appl. Math.},
      volume={52},
      number={2},
       pages={577\ndash 591},
}

\bib{donoho_tsaig}{article}{
      author={Donoho, David~L.},
      author={Tsaig, Yaakov},
       title={Fast solution of $\ell_1$-norm minimization problems when the
  solution may be sparse},
        date={2008},
     journal={IEEE Trans. Inform. Theory},
      volume={54},
       pages={4789\ndash 4812},
}

\bib{ene2019improved}{inproceedings}{
      author={Ene, Alina},
      author={Vladu, Adrian},
       title={Improved convergence for $\ell_1 $ and $\ell\_{\infty} $
  regression via iteratively reweighted least squares},
organization={PMLR},
        date={2019},
   booktitle={{Proceedings of 2019 International Conference on Machine Learning
  (ICML'19)}},
       pages={1794\ndash 1801},
}

\bib{fornasier_peter_rauhut_worm}{article}{
      author={Fornasier, Massimo},
      author={Peter, Steffen},
      author={Rauhut, Holger},
      author={Worm, Stephan},
       title={Conjugate gradient acceleration of iteratively re-weighted least
  squares methods},
        date={2016},
     journal={Comput. Optim. Appl.},
      volume={65},
      number={1},
       pages={205\ndash 259},
}

\bib{fornasier_rauhut_ward}{article}{
      author={Fornasier, Massimo},
      author={Rauhut, Holger},
      author={Ward, Rachel},
       title={Low-rank matrix recovery via iteratively reweighted least squares
  minimization},
        date={2011},
     journal={SIAM J. Optim.},
      volume={21},
       pages={1614\ndash 1640},
}

\bib{FoucartRauhut13}{book}{
      author={Foucart, Simon},
      author={Rauhut, Holger},
       title={A mathematical introduction to compressive sensing},
      series={Applied and Numerical Harmonic Analysis},
   publisher={Springer New York},
        date={2013},
}

\bib{GemanReynolds92}{article}{
      author={Geman, Donald},
      author={Reynolds, George},
       title={{Constrained Restoration and the Recovery of Discontinuities}},
        date={1992},
        ISSN={0162-8828},
     journal={IEEE Trans. Pattern Anal. Mach. Intell.},
      volume={14},
      number={03},
       pages={367\ndash 383},
}

\bib{gorodnitsky_rao}{article}{
      author={Gorodnitsky, Irina},
      author={Rao, Bhaskar},
       title={Sparse signal reconstruction from limited data using focuss: A
  re-weighted minimum norm algorithm},
        date={1997},
     journal={IEEE Trans. Signal Process.},
      volume={45},
      number={3},
       pages={600\ndash 616},
}

\bib{green1984iteratively}{article}{
      author={Green, Peter~J.},
       title={Iteratively reweighted least squares for maximum likelihood
  estimation, and some robust and resistant alternatives},
        date={1984},
     journal={Journal of the Royal Statistical Society: Series B
  (Methodological)},
      volume={46},
      number={2},
       pages={149\ndash 170},
}

\bib{gribonval_nielsen}{article}{
      author={Gribonval, Remi},
      author={Nielsen, Morten},
       title={Sparse representation in the union of basis},
        date={2003},
     journal={IEEE Trans. Inform. Theory},
      volume={49},
       pages={3320\ndash 3325},
}

\bib{hastie2019statistical}{book}{
      author={Hastie, Trevor},
      author={Tibshirani, Robert},
      author={Wainwright, Martin},
       title={Statistical learning with sparsity: The lasso and
  generalizations},
   publisher={Chapman and Hall/CRC},
        date={2019},
}

\bib{HestenesStiefel52}{article}{
      author={Hestenes, Magnus~Rudolph},
      author={Stiefel, Eduard},
       title={Methods of conjugate gradients for solving linear systems},
        date={1952},
     journal={Journal of research of the National Bureau of Standards},
      volume={49},
      number={1},
}

\bib{holland_welsch}{article}{
      author={Holland, Paul~W.},
      author={Welsch, Roy~E.},
       title={Robust regression using iteratively reweighted least-squares},
        date={1977},
     journal={Communications in Statistics - Theory and Methods},
      volume={6},
       pages={813\ndash 827},
}

\bib{Huber64}{article}{
      author={Huber, Peter~J.},
       title={Robust estimation of a location parameter},
        date={1964},
     journal={Ann. Math. Statist.},
      volume={35},
      number={1},
       pages={73\ndash 101},
}

\bib{idier}{article}{
      author={Idier, J\'er\^ome},
       title={Convex half-quadratic criteria and interacting auxiliary
  variables for image restoration},
        date={2001},
     journal={IEEE Trans. Image Process.},
      volume={10},
      number={7},
       pages={1001\ndash 1009},
}

\bib{KM21}{inproceedings}{
      author={K\"ummerle, Christian},
      author={Mayrink~Verdun, Claudio},
       title={A scalable second order method for ill-conditioned matrix
  completion from few samples},
        date={2021},
   booktitle={{Proceedings of 2021 International Conference on Machine Learning
  (ICML'21)}},
}

\bib{kummerle_sigl}{article}{
      author={K\"ummerle, Christian},
      author={Sigl, Juliane},
       title={Harmonic mean iteratively reweighted least squares for low-rank
  matrix recovery},
        date={2018},
     journal={J. Mach. Learn. Res.},
      volume={19},
       pages={1815\ndash 1863},
}

\bib{lai_yangyang_wotao}{article}{
      author={Lai, Ming-Jun},
      author={Yangyang, Xu},
      author={Wotao, Yin},
       title={Improved iteratively reweighted least squares for unconstrained
  smoothed $\ell_q$ minimization},
        date={2013},
     journal={SIAM J. Numer. Anal.},
      volume={2},
      number={51},
       pages={927\ndash 957},
}

\bib{lerman2018fast}{article}{
      author={Lerman, Gilad},
      author={Maunu, Tyler},
       title={Fast, robust and non-convex subspace recovery},
        date={2018},
     journal={Information and Inference: A Journal of the IMA},
      volume={7},
      number={2},
       pages={277\ndash 336},
}

\bib{li_sun_toh}{article}{
      author={Li, Xudong.},
      author={Sun, Defeng},
      author={Toh, Kim-Chuan},
       title={A highly efficient semismooth newton augmented lagrangian method
  for solving lasso problems},
        date={2018},
     journal={SIAM J. Optim.},
      volume={28},
      number={1},
       pages={433\ndash 458},
}

\bib{liang2017activity}{article}{
      author={Liang, Jingwei},
      author={Fadili, Jalal},
      author={Peyr{\'e}, Gabriel},
       title={Activity identification and local linear convergence of
  forward--backward-type methods},
        date={2017},
     journal={SIAM Journal on Optimization},
      volume={27},
      number={1},
       pages={408\ndash 437},
}

\bib{meinshausen2006high}{article}{
      author={Meinshausen, Nicolai},
      author={B{\"u}hlmann, Peter},
      author={others},
       title={High-dimensional graphs and variable selection with the lasso},
        date={2006},
     journal={Annals of statistics},
      volume={34},
      number={3},
       pages={1436\ndash 1462},
}

\bib{Meurant}{book}{
      author={Meurant, G{\'e}rard},
       title={The lanczos and conjugate gradient algorithms: From theory to
  finite precision computations},
   publisher={Society for Industrial and Applied Mathematics,},
        date={2006},
}

\bib{meyer2021alternative}{inproceedings}{
      author={Meyer, Gregory~P},
       title={An alternative probabilistic interpretation of the huber loss},
        date={2021},
   booktitle={Proceedings of the ieee/cvf conference on computer vision and
  pattern recognition},
       pages={5261\ndash 5269},
}

\bib{mohan_fazel}{article}{
      author={Mohan, Karthik},
      author={Fazel, Maryam},
       title={Iterative reweighted algorithms for matrix rank minimization},
        date={2012},
     journal={J. Mach. Learn. Res.},
      volume={13},
       pages={3441\ndash 3473},
}

\bib{molinari2021iterative}{inproceedings}{
      author={Molinari, Cesare},
      author={Massias, Mathurin},
      author={Rosasco, Lorenzo},
      author={Villa, Silvia},
       title={Iterative regularization for convex regularizers},
organization={PMLR},
        date={2021},
   booktitle={International conference on artificial intelligence and
  statistics},
       pages={1684\ndash 1692},
}

\bib{bhaskar_govind_prateek_purushottam}{inproceedings}{
      author={Mukhoty, Bhaskar},
      author={Gopakumar, Govind},
      author={Jain, Prateek},
      author={Kar, Purushottam},
       title={Globally-convergent iteratively reweighted least squares for
  robust regression problems},
        date={2019},
   booktitle={{Proceedings of Machine Learning Research}},
      volume={89},
       pages={313\ndash 322},
}

\bib{natarajan}{article}{
      author={Natarajan, Balas~K.},
       title={Sparse approximate solutions to linear systems},
        date={1996},
     journal={SIAM J. Comput.},
      volume={24},
       pages={227\ndash 234},
}

\bib{NikolavaNg05}{article}{
      author={Nikolova, Mila},
      author={Ng, Michael~K.},
       title={Analysis of half-quadratic minimization methods for signal and
  image recovery},
        date={2005},
     journal={SIAM J. Sci. Comput.},
      volume={27},
      number={3},
       pages={937\ndash 966},
}

\bib{NocedalWright06}{book}{
      author={Nocedal, Jorge},
      author={Wright, Stephen},
       title={Numerical optimization},
   publisher={Springer Science \& Business Media},
        date={2006},
}

\bib{ochs_dosovitskiy_brox_pock}{article}{
      author={Ochs, Peter},
      author={Dosovitskiy, Alexey},
      author={Brox, Thomas},
      author={Pock, Thomas},
       title={On iteratively reweighted algorithms for nonsmooth nonconvex
  optimization in computer vision},
        date={2015},
     journal={SIAM J. Imaging Sci.},
      volume={8},
      number={1},
       pages={331\ndash 372},
}

\bib{2021-Poon-noncvxpro}{inproceedings}{
      author={Poon, Clarice},
      author={Peyr{\'e}, Gabriel},
       title={{Smooth Bilevel Programming for Sparse Regularization}},
        date={2021},
   booktitle={{Proc. NeurIPS'21}},
         url={https://arxiv.org/abs/2106.01429},
}

\bib{rao_kreutz-delgado}{article}{
      author={Rao, Bhaskar~D.},
      author={Kreutz-Delgado, Kenneth},
       title={An affine scaling methodology for best basis selection},
        date={1999},
     journal={IEEE Trans. Signal Process.},
      volume={47},
      number={1},
       pages={187\ndash 200},
}

\bib{rudelson_vershynin}{article}{
      author={Rudelson, Mark},
      author={Vershynin, Roman},
       title={On sparse reconstruction from fourier and gaussian measurements},
        date={2008},
     journal={Comm. Pure Appl. Math.},
      volume={61},
      number={8},
       pages={1025\ndash 1045},
}

\bib{spielman_teng}{article}{
      author={Spielman, Daniel~A.},
      author={Teng, Shang-Hua},
       title={Smoothed analysis of termination of linear programming
  algorithms},
        date={2003},
     journal={Math. Program., Ser. B},
      volume={97},
       pages={375\ndash 404},
}

\bib{straszak2021iteratively}{article}{
      author={Straszak, Damian},
      author={Vishnoi, Nisheeth~K},
       title={Iteratively reweighted least squares and slime mold dynamics:
  connection and convergence},
        date={2021},
     journal={Mathematical Programming},
       pages={1\ndash 33},
}

\bib{sun_babu_palomar}{article}{
      author={Sun, Ying},
      author={Babu, Prabhu},
      author={Palomar, Daniel~P.},
       title={Majorization-minimization algorithms in signal processing,
  communications, and machine learning},
        date={2017},
     journal={IEEE Trans. Signal Process.},
      volume={65},
      number={3},
       pages={794\ndash 816},
}

\bib{TaylorBankMcCoy79}{article}{
      author={Taylor, Howard~L.},
      author={Banks, Stephen~C.},
      author={McCoy, John~F.},
       title={Deconvolution with the $\ell_1$ norm},
        date={1979},
     journal={Geophysics},
      volume={44},
      number={1},
       pages={39\ndash 52},
}

\bib{tibshirani1996regression}{article}{
      author={Tibshirani, Robert},
       title={Regression shrinkage and selection via the lasso},
        date={1996},
     journal={Journal of the Royal Statistical Society: Series B
  (Methodological)},
      volume={58},
      number={1},
       pages={267\ndash 288},
}

\bib{tillmann2015equivalence}{article}{
      author={Tillmann, Andreas~M},
       title={Equivalence of linear programming and basis pursuit},
        date={2015},
     journal={PAMM},
      volume={15},
      number={1},
       pages={735\ndash 738},
}

\bib{Voronin12}{book}{
      author={Voronin, Sergey},
       title={Regularization of linear systems with sparsity constraints with
  applications to large scale inverse problems},
   publisher={Ph.D. Thesis, Princeton University},
        date={2012},
}

\bib{voronin_daubechies}{incollection}{
      author={Voronin, Sergey},
      author={Daubechies, Ingrid},
       title={An iteratively reweighted least squares algorithm for sparse
  regularization},
        date={2017},
   booktitle={Functional analysis, harmonic analysis, and image processing: A
  collection of papers in honor of bj{\"o}rn jawerth},
   publisher={American Mathematical Society},
       pages={391\ndash 441},
}

\bib{weiszfeld}{article}{
      author={Weiszfeld, Endre},
       title={Sur le point pour lequel la somme des distances de n points
  donn{\'e}s est minimum},
        date={1937},
     journal={Tohoku Mathematical Journal, First Series},
      volume={43},
       pages={355\ndash 386},
}

\bib{Woodbury50}{article}{
      author={Woodbury, Max~A.},
       title={Inverting modified matrices},
        date={1950},
     journal={Memorandum report},
      volume={42},
      number={106},
       pages={336},
}

\bib{yin_osher_goldfarb_darbon}{article}{
      author={Yin, Wotao},
      author={Osher, Stanley},
      author={Goldfarb, Donald},
      author={Darbon, Jerome},
       title={Bregman iterative algorithms for $\ell_1$-minimization with
  applications to compressed sensing},
        date={2008},
     journal={SIAM J. Imaging Sci},
      volume={1},
      number={1},
       pages={143\ndash 168},
}

\bib{zhang2019joint}{inproceedings}{
      author={Zhang, Zaiwei},
      author={Huang, Xiangru},
      author={Huang, Qixing},
      author={Zhang, Xiao},
      author={Li, Yuan},
       title={Joint learning of neural networks via iterative reweighted least
  squares.},
        date={2019},
   booktitle={Cvpr workshops},
       pages={18\ndash 26},
}

\end{biblist}
\end{bibdiv}

\newpage

\appendix

\section*{Supplementary material of \emph{Iteratively Reweighted Least Squares for Basis Pursuit with Global Linear Convergence Rate}}

In this supplement to the paper, we present in Section \ref{app:approxsparse} our second main result, Theorem \ref{mainresult:approximatesparse}, which generalizes the global linear convergence of \Cref{alg:algo1} for sparse vectors, which we presented Theorem \ref{mainresult:sparse}, to approximately sparse vectors. 
In Section \ref{section:proofs} we provide the proof of our theoretical results. In Section \ref{section:ineq_solution}, we discuss how to solve the (weighted) least squares step in Algorithm \ref{alg:algo1} in a computationally efficient manner. 
Finally, Section \ref{sec:proof_technical_lemmas} contains, for completeness, a proof of two technical results, namely Lemma \ref{lemma:IRLS:basicproperties} and Proposition \ref{prop:DDFG:supportcondition}.

\etocdepthtag.toc{mtappendix}
\etocsettagdepth{mtchapter}{none}
\etocsettagdepth{mtappendix}{subsection}
\tableofcontents

\section{Linear Convergence of IRLS for approximately sparse vectors} \label{app:approxsparse}
We now generalize Theorem \ref{mainresult:sparse} to the scenario where the ground truth $x_{\ast}$ is only approximately sparse. 
By that, we mean that the vector $x_{\ast}$ can be well-approximated by an $s$-sparse vector in the sense that the $\ell_1$-error of the best $s$-term approximation $\sigma_s(x_{\ast})_{\ell_1}=\inf \{\Vert x_{\ast}-z \Vert_1: \ z \in \mathbb{R}^N \ \text{is $s$-sparse} \}$ is small, which is a commonly used quantity to measure the model misfit to a sparse vector \cite[Section 2.1]{FoucartRauhut13}.
If $x_{\ast}$ is approximately sparse in this sense, we can only hope to \emph{approximately} recover $x_{\ast}$ by the $\ell_1$-minimization program \eqref{basis_pursuit}. Indeed, \cite[Theorem 5.3(iv)]{Daubechies10} showed that under a suitable null space property, IRLS for \eqref{basis_pursuit} finds a vector $x$, such that $\onenorm{x-x_{\ast}} $ is at most a constant multiple of the optimal best $s$-term approximation error $\sigma_s(x_{\ast})_{\ell_1}$. However, as for exactly sparse vectors $x_{\ast}$, only a local, but no global convergence rate was provided in previous literature \cite[Theorem 6.4]{Daubechies10}.

The following result shows that in fact, we also can obtain global linear convergence of \Cref{alg:algo1} in this case. More precisely, \Cref{mainresult:approximatesparse} implies that $ \mathcal{J}_{\varepsilon_{k}}(x^{k}) -  \onenorm{x_{*}}$ decays exponentially fast until a certain accuracy is reached, which $\sigmaxstar$ up to a constant multiple.
\begin{theorem}\label{mainresult:approximatesparse}
Consider the problem of recovering an unknown vector $ x_{*} \in \R^N $ from known measurements of the form $y=Ax_{*} $. 
 Assume that the measurement matrix $A \in \mathbb{R}^{m \times N}$ fulfills the $\ell_1$-NSP of order $s$ with constant $\rho_s < 1/8$. Let the IRLS iterates  $\left\{    x^{k} \right\}_k$ and $ \left\{ \varepsilon_{k} \right\}_k$ be defined by \eqref{step_1} and \eqref{step_2} with initialization $x^0$. Then the following three statements hold.
\begin{enumerate}
\item  For  $k  \le  \hat{k}:=  \min \left\{ k \in \mathbb{N} :  \sigmaxstar> \frac{2}{9}\left\|\left(x_{*}  -x^{k} \right)_{S^c}\right\|_1   \right\}$ it holds that
\begin{equation} \label{equ:linearconvergence2} 
\mathcal{J}_{\varepsilon_{k}}(x^{k}) -  \onenorm{x_{*}} \le \left(  1-  \frac{c}{\rho_1 N}   \right)^k \left(     \mathcal{J}_{\varepsilon_{0}}(x^0) -  \onenorm{x_{*}} \right),
\end{equation}
where $S$ denotes the support of the $s$ largest entries of $x_{\ast}$.
\item For all $1\le k \le \hat{k} $ it holds that
\begin{equation}\label{ineq:approxsparse1}
\onenorm{x^k -x_{*}}    \le  6 \left(  1-  \frac{c}{\rho_1 N}   \right)^k  \onenorm{x^0 -x_{*}} + 10  \sigmaxstar.
\end{equation}
\item Moreover, for all integers $ k \gtrsim \rho_1 N  \log \left(   \frac{ \onenorm{x^0 -x_{*}} }{  \sigmaxstar }    \right)  $ we have that
\begin{equation}\label{ineq:approxsparse2}
	\onenorm{x^k -x_{*}}    \le 20  \sigmaxstar .
\end{equation}

\end{enumerate}
Here $ c=1/3072$ and $\rho_1 < 1/8$ denotes the constant for the $\ell_1$-NSP of order $1$.
\end{theorem}

\begin{remark}
Applying \Cref{mainresult:approximatesparse} to the special case $ \sigmaxstar=0 $, we observe that inequality \eqref{ineq:globalconvergenerate1} yields a seemingly sharper result than inequality \eqref{ineq:approxsparse1} in \Cref{mainresult:sparse}, which may seem somewhat counterintuitive.
However, note that in \Cref{mainresult:sparse} we require $\rho_s < 1/2 $, whereas in \Cref{mainresult:approximatesparse} we have the stronger assumption $ \rho_s < 1/8 $. Indeed, a closer inspection of the proofs reveals that both the factors $3$ and $6$  in the inequalities \eqref{ineq:globalconvergenerate1} and \eqref{ineq:approxsparse1} can be replaced by the factor $ \frac{3(1+\rho_s)}{1-\rho_s}$, reconciling those two results.
\end{remark}

\section{Proofs of the main results}\label{section:proofs}
In this section, we prove the main results of this paper, \Cref{mainresult:sparse}, and its approximately sparse counterpart, \Cref{mainresult:approximatesparse}. To this end, we first state and prove the following technical lemma, which gives an upper and lower bound for $\mathcal{J}_{\varepsilon} \left(x\right) - \onenorm{x_{*}}$, which is the quantity for which we are going to show linear convergence.
\begin{lemma}\label{lemma:epscontrol} Let $x_{*},  x \in \mathbb{R}^N$. Assume that $A$ fulfills the $\ell_1$-NSP of order $s$ with constant $\rho_s <1$. Furthermore, suppose $Ax_{*}= Ax$ and that $\varepsilon \le \frac{1}{N} \sigma_s \left( x \right)_{\ell_1}$. Then it holds that
	\begin{equation}\label{ineq:aux1}
	\frac{1-\rhos}{1+\rhos}  \onenorm{x -x_{*}}  -  2  \sigmaxstar \le 
	\mathcal{J}_{\varepsilon} \left(x\right) - \onenorm{x_{*}} \le   3 \sigma_s \left(x\right)_{\ell_1} .
	\end{equation}
\end{lemma} 
In order to prove \Cref{lemma:epscontrol} we need the following technical lemma.
\begin{lemma}\label{lemma:NSPl1min}\cite[Lemma 4.3]{Daubechies10}
	Assume that the matrix $A \in \mathbb{R}^{m \times N} $ fulfills the $\ell_1$-NSP for some $s$ and $\rho_s<1$. Then  for all $x, x_{*} \in \mathbb{R}^N$ such that $Ax = Ax_{*} $ it holds that
	
	\[
	\label{ball17}
	\|x-x_{*}\|_{\ell_1} 
	\leqslant \frac{1+\rho_s}{1-\rho_s}\left(\|x_{*}\|_1 - \|x\|_1 + 2 \sigma_s(z)_{\ell_1}\right).
	\]
	
\end{lemma}

\begin{proof}[Proof of \Cref{lemma:epscontrol}]
	We observe that $ \mathcal{J}_{\varepsilon} \left(x\right) \ge \onenorm{x}$ for each $x \in \R^{N}$, which follows directly from the definition of $ \mathcal{J}_{\varepsilon} \left(x\right) $, see \Cref{eq:smoothedell1:objective}. Now, fix $x_{*}, x \in \R^{N}$, and let $S$ be the set which contains the $s$ largest entries of $x$ in absolute value. Hence, we obtain that
	\begin{align*}
	\mathcal{J}_{\varepsilon} \left(x\right) - \onenorm{x_{*}} &\ge \onenorm{x} - \onenorm{x_{*}} \\
	&= \onenorm{x_{S^c}} + \onenorm{x_{S}}   - \onenorm{x_{*}} \\
	&\ge \onenorm{x_{S^c}} - \onenorm{\left(x-x_{*}\right)_S}-\onenorm{ \left( x_{*} \right)_{S^c} }\\
	&\ge \onenorm{\left( x-x_{*}  \right)_{S^c}} - \onenorm{\left(x-x_{*}\right)_S}-  2 \onenorm{ \left( x_{*} \right)_{S^c} },
	\end{align*} 
		where in each of the last two inequalities we have applied the reverse triangle inequality. 
	Since $x-x_{*}$ is contained in the null space of $A$, it follows from the nullspace property that $\onenorm{\left(x-x_{*}\right)_S} \le \rhos \onenorm{\left(x-x_{*}\right)_{S^c}}$. Hence, we have shown that
	\begin{align*}
	\mathcal{J}_{\varepsilon} \left(x\right) - \onenorm{x_{*}} &\ge \left(1-\rhos\right) \onenorm{\left( x-x_{*}  \right)_{S^c}} -    2 \onenorm{ \left( x_{*} \right)_{S^c} }.
	\end{align*}
Since it follows from the null space property that $ \onenorm{\left( x-x_{*}  \right)_{S^c}} \ge \frac{\onenorm{x -x_{*}}}{1+\rhos} $, this shows the first inequality in \eqref{ineq:aux1}.
	
	Next, we are going to prove the reverse inequality in \eqref{ineq:aux1}.  For that, set $I := \{i \in [N]: |x_i| > \varepsilon\}$. Then we observe that
	\begin{equation} \label{eq:Jupperbd}
	\begin{split}
	\mathcal{J}_{ \varepsilon} \left(x\right)  -\onenorm{x_{*}} &= \onenorm{x_I} + \frac{1}{2} \sum_{i \in I^c} \left( \frac{x_i^2}{\varepsilon}    + \varepsilon  \right)  -\onenorm{x_{*}}\\
	&\le \onenorm{x_I} +  \vert I^c \vert \varepsilon -\onenorm{x_{*}} \\
	&\le  \onenorm{x_I}  +  \sigma_s \left( x \right)_{\ell_1} -\onenorm{x_{*}}\\
	&\le  \onenorm{x}  +  \sigma_s \left( x \right)_{\ell_1}-\onenorm{x_{*}}.
	\end{split}
	\end{equation}
	In the third line we used the  assumption $\varepsilon \le \frac{1}{N} \sigma_s \left( x \right)_{\ell_1}$.
	In order to proceed, we first derive an appropriate upper bound for $ \onenorm{x}  -\onenorm{x_{*}} $. For that, we note
	\begin{align*}
	\left(    \frac{1-\rhos}{1+\rhos}+1\right)\left( \onenorm{x} -\onenorm{x_{*}} \right)   &\leq \frac{1-\rhos}{1+\rhos}\|x-x_{*}\|_1- \left(\onenorm{x_{*}}  - \onenorm{x} \right) \\
	&\leq  \left(\onenorm{x_{*}}  - \onenorm{x}  + 2 \sigma_s \left(x\right)_{\ell_1}  \right) - \left(\onenorm{x_{*}}  - \onenorm{x} \right) \\
	&\leq 2 \sigma_s \left(x\right)_{\ell_1} ,
	\end{align*}
	where in the second line we have used \Cref{lemma:NSPl1min}. This shows that $\onenorm{x} -\onenorm{x_{*}}  \le \frac{2 \sigma_s \left(x\right)_{\ell_1} }{1+ \frac{1-\rhos}{1+\rhos}} $. Combining this with \eqref{eq:Jupperbd}, we obtain
	\begin{align*}
	\mathcal{J}_{\varepsilon} \left(x\right)  - \onenorm{x_{*}}
	& \le  3 \sigma_s \left(x\right)_{\ell_1} ,
	\end{align*}
	which finishes the proof of inequality \eqref{ineq:aux1}.
\end{proof}
The next key proposition states that the quantity $\mathcal{J}_{\varepsilon_k} \left(x^k\right) - \onenorm{x_{*}} $ decays linearly under appropriate conditions.
\begin{proposition}\label{thm:p1:linearrate}
Let $x_{*} \in \R^N$ be an approximately $s$-sparse vector with support $S$. Let $A \in \R^{m \times N}$ and $y=Ax_{*}$. Assume that $A$ fulfills the $\ell_1$-NSP of order $s$ with constant $ \rho_s < 3/4$, if  $\sigmaxstar =0  $, and $ \rho_s < 1/4 $ otherwise. Denote by $\rho_1 $ the NSP constant of order 1.\\
Let the IRLS iterates  $\left\{    x^{k} \right\}_k$ and $ \left\{ \varepsilon_{k} \right\}_k$ be defined by \eqref{step_1} and \eqref{step_2} with initialization $x^0$.
Then, for all $k \in \mathbb{N}$, such that $ \onenorm{ \left( x_{*} \right)_{S^c}   } \le  \frac{2}{9}\onenorm{\left(x_{*}\right)_{S^c} -x^{\ell}_{S^c}} $ for all  $ \ell < k $, the following holds
\begin{equation} \label{eq:IRLS:suffdecrease}
\mathcal{J}_{\varepsilon_{k}}(x^{k}) -  \onenorm{x_{*}} \le \left(  1-  \frac{c_{\rhos}}{\rho_1 N}   \right)^k \left(     \mathcal{J}_{\varepsilon_{0}}(x^0) -  \onenorm{x_{*}} \right). 
	\end{equation}
where the constant $c_{\rhos}$ is defined by 
\begin{equation*}
c_{\rhos}:=
\begin{cases}
  \frac{(3/4-\rhos)^2}{48}  \quad & \text{if  } \sigmaxstar =0 \\
    \frac{(1/4-\rhos)^2}{48} \quad & \text{else}
 \end{cases}
\end{equation*}

\end{proposition}
Before proving this statement, let us describe the main ideas of our proof. 
Recall that $x_{*}$ has minimal $\Vert \cdot \Vert_1$-norm among all vectors $x$, which satisfy the constraint $Ax=y$. Hence, setting $v^k = x_{*} - x^k $ due to convexity of the $\ell_1$-norm we that $\onenorm{x^k + t v^k} < \onenorm{x^k}$ for all $0<t<1$.  
Since that the quadratic functional $	Q(\cdot, x^k)$ approximates the objective function $ \mathcal{J}_{\varepsilon} $, which is a surrogate for the $\ell_1$-norm, in a neighborhood of the current iterate $x^k $, we also expect that for $t>0$ sufficiently small we have that $Q(x^k + t v^k, x^k) < Q(x^k, x^k) $.
In order to show that the decrease is sufficiently large, we also need to show that $t$ can be chosen large enough.
This will guarantee a sufficient decrease of $\mathcal{J}_{\varepsilon_k} \left(x^k\right) $ in each iteration.

\begin{proof}[Proof of  \Cref{thm:p1:linearrate}]
In order to show inequality \eqref{eq:IRLS:suffdecrease} we will prove by induction that for each $k$,  such that $ \onenorm{ \left( x_{*} \right)_{S^c}   } \le  \frac{2}{9}\onenorm{\left(x_{*}\right)_{S^c} -x^{\ell}_{S^c}} $ for all  $ \ell \le  k $, it holds that
\begin{equation*}
 \mathcal{J}_{\varepsilon_{k+1}}(x^{k+1}) -  \onenorm{x_{*}}    \le \left( 1 - \frac{c_{\rhos}}{\rho_1 N}  \right) \left( \mathcal{J}_{\varepsilon_{k}}(x^{k}) -  \onenorm{x_{*}} \right).
\end{equation*}
Now choose such a $k \ge 1$ and assume that the statement has been shown for all $ k' <k$.
Set $ v^k = x_{*} - x^k $.
For $t \in \R$, we have, by optimality of $x^{k+1}$ in \eqref{step_1}, that
\begin{equation} \label{eq:mainidea}
 \mathcal{J}_{\varepsilon_{k+1}}(x^{k+1}) \leq Q_{\varepsilon_k} (x^{k+1},x^k) \leq Q_{\varepsilon_k}(x^k+t v^k,x^k).
\end{equation}
Moreover, by the definition of the quadratic objective $Q_{\varepsilon_k}(\cdot,x^k)$ (see \eqref{eq:smoothedell1:IRLSmajorizer}), it holds that
\begin{equation} \label{eq:objective:diff}
\begin{split}
Q_{\varepsilon_k}(x^k+t v^k,x^k) - \mathcal{J}_{\varepsilon_k}(x^k) &=  t \, \langle \nabla \mathcal{J}_{\varepsilon_k}(x^k), v^k\rangle + \frac{t^2}{2} \langle v^k, \diag(w(x^k,\varepsilon_k)) v^k \rangle.
\end{split}
\end{equation}	
Our goal is to show that by picking $t$ large enough, we can make $Q_{\varepsilon_k}(x^k+t v^k,x^k) - \mathcal{J}_{\varepsilon_k}(x^k)<0$ sufficiently small. For that, we now control $\langle \nabla \mathcal{J}_{\varepsilon_k}(x^k), v^k\rangle$ and $\langle v^k, \diag(w_{\varepsilon_k}(x^k)) v^k \rangle$ separately.\\

\noindent \textbf {Part 1: Bounding the linear term $\langle \nabla \mathcal{J}_{\varepsilon_k}(x^k), v^k\rangle$:} \label{sec:bound:linear}

Let $I := \{i \in [N]: |x^k_i| > \varepsilon_k\}$ and denote by $S$ the set which contains the $s$ largest entries  of $x_{*}$ in absolute value. In the case that $ x_{*} $ is sparse, $S$ is given by the support of $x_{*}$, i.e. $ S = \text{ supp} \left(x_{*}\right) $.  Consider

\begin{align*}
\langle \nabla \mathcal{J}_{\varepsilon_k}(x^k), v^k\rangle  &= \sum_{i=1}^N \frac{x_i^{k}}{\max(|x_i^{k}|,\varepsilon_k)} v^k_i = \sum_{i \in S} \frac{x_i^{k}}{\max(|x_i^{k}|,\varepsilon_k)} v^k_i + \sum_{i \in S^c} \frac{x_i^{k}}{\max(|x_i^{k}|,\varepsilon_k)} v^k_i.
\end{align*} 
The first summand can be bounded by
\begin{align*}
\sum_{i \in S} \frac{x_i^{k}}{\max(|x_i^{k}|,\varepsilon_k)} v^k_i &= \sum_{i \in S \cap I} \sgn(x_i^{(k)}) v^k_i + \sum_{i \in S \cap I^c} \frac{x_i^{k}}{\varepsilon_k} v^k_i  \\
&\leq \|v^k_{S \cap I}\|_1 + \|v^k_{S \cap I^c}\|_1 \\
&= \|v^k_S\|_1\\
&\leq  \rho_s\|v^k_{S^c}\|_1.
\end{align*}

Where the last inequality comes from the NSP for $v^k$ which, in turns, comes from the fact that $v^k=x_*-x^k \in \ker(A)$, since $Ax_*=Ax^k$. For the second summand we have that
\begin{align*}
&\sum_{i \in S^c} \frac{x_i^{k}}{\max(|x_i^{k}|,\varepsilon_k)} v^k_i \\
=& \sum_{i \in S^c \cap I} \text{sgn} \left(x_i^{k}\right)   v^k_i   + \sum_{i \in S^c \cap I^c} \frac{ x_i^{k} v^k_i}{\varepsilon_k} \\
=& \sum_{i \in S^c \cap I} \text{sgn} \left(x_i^{k}\right)   \left(x_{*}\right)_i  -   \sum_{i \in S^c \cap I} \text{sgn} \left(x_i^{k}\right)   x_i^{k}   +  \sum_{i \in S^c \cap I^c} \frac{ x_i^{k} \left(x_{*}\right)_i}{\varepsilon_k} - \sum_{i \in S^c \cap I^c} \frac{  (x_i^k)^2}{\varepsilon_k}  \\
\le&\onenorm{ \left(  x_{*} \right)_{S^c \cap I } }   - \onenorm{ x^k_{S^c \cap I} } +  \onenorm{ \left(  x_{*} \right)_{S^c \cap I^c } } - \frac{\twonorm{x^k_{S^c \cap I^c}}^2}{\varepsilon_k} \\
=&  - \onenorm{ x^k_{S^c \cap I} } +  \onenorm{ \left(  x_{*} \right)_{S^c}} - \frac{\twonorm{x^k_{S^c \cap I^c}}^2}{\varepsilon_k} \\
= &  \onenorm{ \left(  x_{*} \right)_{S^c}  } - \onenorm{ x^k_{S^c} } + \onenorm{ x^k_{S^c \cap I^c } }  - \frac{\twonorm{x^k_{S^c \cap I^c}}^2}{\varepsilon_k}\\
\le&  2 \onenorm{ \left(  x_{*} \right)_{S^c}  } - \onenorm{ v^k_{S^c } } + \onenorm{ x^k_{S^c \cap I^c } }  - \frac{\twonorm{x^k_{S^c \cap I^c}}^2}{\varepsilon_k}.
\end{align*}
In order to proceed, we note that from the elementary inequality $ab \le \frac{1}{2} \left( a^2 + b^2\right) $  and from $\onenorm{ x^k_{S^c \cap I^c}} \leq \sqrt{N} \twonorm{x^k_{S^c \cap I^c}}$, it follows that
\begin{align*}
\onenorm{ x^k_{S^c \cap I^c } } &\le \frac{1}{2} \left( \frac{\varepsilon_k  \onenorm{ x^k_{S^c \cap I^c } }^2}{2 \twonorm{x^k_{S^c \cap I^c}}^2}         +2  \frac{\twonorm{x^k_{S^c \cap I^c}}^2}{\varepsilon_k}. \right) \le  \frac{\varepsilon_k N}{4} + \frac{\twonorm{x^k_{S^c \cap I^c}}^2}{\varepsilon_k}.
\end{align*}
Hence, using that $\varepsilon_k \leq \sigma_s(x^k)_{\ell_1}/N$, we have shown that
\begin{align*}
\sum_{i \in S^c} \frac{x_i^{k}}{\max(|x^k_i|,\varepsilon_k)} v^k_i  &\le 2 \onenorm{ \left(  x_{*} \right)_{S^c}  } - \onenorm{ v^k_{S^c } }  +  \frac{\varepsilon_k N}{4}\\
& \le 2 \onenorm{ \left(  x_{*} \right)_{S^c}  } - \onenorm{ v^k_{S^c } }  +  \frac{ \sigma_s(x^k)_{\ell_1} }{4} \\
&\le 2 \onenorm{ \left(  x_{*} \right)_{S^c}  } - \onenorm{ v^k_{S^c } }  +  \frac{ \onenorm{x^k_{S^c}} }{4}\\
& \le   2 \onenorm{ \left(  x_{*} \right)_{S^c}  } - \onenorm{ v^k_{S^c } }  +  \frac{ \onenorm{v^k_{S^c}} }{4}+ \frac{ \onenorm{(x_{*})_{S^c}} }{4}
\\& =   \frac{9}{4}  \onenorm{ \left(  x_{*} \right)_{S^c}  } - \frac{3}{4} \onenorm{ v^k_{S^c } },
\end{align*}
where we used the triangular inequality for the vector $v^k=x^k-x_{*} $ on the set $S^c$ and the fact that $\sigma_s(x^k)_{\ell_1} \leq \onenorm{ x^k_{S^c}}$. Hence, by adding up terms we obtain that
\begin{align*}
\langle \nabla \mathcal{J}_{\varepsilon_k}(x^k), v^k\rangle &\le \frac{9}{4}  \onenorm{ \left(  x_{*} \right)_{S^c}  } -   \left(\frac{3}{4}   -\rho_s   \right) \onenorm{ v^k_{S^c } } \leq -( \beta -\rhos) \onenorm{ v^k_{S^c }}.
\end{align*}
Here, we have set $\beta=3/4$ in the case that $\sigmaxstar =0 $ and $ \beta =1/4$ else. Moreover, we used the assumption  $\onenorm{ \left( x_{*} \right)_{S^c}   } \le  \frac{2}{9}\onenorm{v^k_{S^c}}$. \\

\noindent \textbf{Part II: Bounding the quadratic term $\langle v^k, \diag(w_{\varepsilon_k}(x^k))v^k \rangle$} \label{sec:bound:quadratic}

\noindent In order bound the quadratic term in \eqref{eq:objective:diff} we first decompose it into two parts
\begin{equation} \label{eq:bound:quadraticterm_0}
\begin{split}
\langle v^k, \diag(w_{\varepsilon_k}(x^k)) v^k \rangle &= \sum_{i=1}^{N} \frac{(v^k_i)^2}{\max(|x_i^{k}|,\varepsilon_k)}  = \sum_{i \in S } \frac{(v^k_i)^2}{\max(|x_i^{k}|,\varepsilon_k)} + \sum_{i \in S^c } \frac{(v^k_i)^2}{\max(|x_i^{k}|,\varepsilon_k)}.
\end{split}
\end{equation}
For the first summand, we note that
\begin{align}\label{quadraticterm_S}
\sum_{i \in S}  \frac{ (v^k_i)^2}{\max(|x^k_i|,\varepsilon_k)}   \le \frac{\Vert v^k_S \Vert_1 \Vert  v^k_S \Vert_{\infty}}{\varepsilon_k} \le  \rho_s \frac{\Vert  v^k_{S^c} \Vert_1  \Vert v^k \Vert_{\infty}  }{\varepsilon_k} \le \frac{\Vert  v^k_{S^c} \Vert_1  \Vert v^k \Vert_{\infty}  }{\varepsilon_k}.
\end{align}

\noindent For the second summand, it holds that
\begin{align}\label{quadraticterm_Sc}
\sum_{i \in S^c}  \frac{ \left(  v^k_i  \right)^2  }{\max(|x^k_i|,\varepsilon_k) }  \le \frac{\Vert  v^k_{S^c} \Vert_{\infty}   \onenorm{ v^k_{S^c} }  }{\varepsilon_k} \leq  \frac{\onenorm{   v^k_{S^c}}  \Vert v^k \Vert_{\infty}     }{\varepsilon_k},
\end{align}
Hence, by adding \eqref{quadraticterm_S} and \eqref{quadraticterm_Sc} up, it follows that
\begin{align*}
\langle  v^k, \diag(w_{\varepsilon_k}(x^k))  v^k \rangle
\le  2 \frac{\onenorm{   v^k_{S^c}}  \Vert v^k \Vert_{\infty}     }{\varepsilon_k}.
\end{align*}
Next, we note that
\begin{equation*}
\Vert v^k \Vert_{\infty}  \le \rho_1 \onenorm{ v^k } \le  \rho_1 \left(1+\rho_s\right) \onenorm{ v^k_{S^c} } \le 2 \rho_1  \onenorm{ v^k_{S^c} }.
\end{equation*}
Hence, we have shown that
\begin{equation*}
\langle  v^k, \diag(w_{\varepsilon_k}(x^k))  v^k \rangle\le  4 \rho_1  \frac{\Vert  v^k_{S^c} \Vert_1^2}{\varepsilon_k}  .
\end{equation*}

\noindent \textbf{Part III: Combining the bounds to obtain decrease in $k$-th step:}

Inserting the bounds of Part I and Part II into \eqref{eq:objective:diff} we obtain
\begin{equation} \label{eq:bound:termdifference_1}
Q_{\varepsilon_k}(x^k+t v^k,x^k) - \mathcal{J}_{\varepsilon_k}(x^k) \leq  - t b + t^2 a =:h(t),
\end{equation}
where the function $h: \R \to \R$ is a quadratic polynomial with coefficients $ b = \left( \beta - \rho_s\right) \|v^k_{S^c}\|_1$ and $ a = 4\rho_1 \frac{\Vert v^k_{S^c} \Vert_1^2}{\varepsilon_k}  $. We observe that the minimizer of $h$ is given by $t= \frac{b}{2a}$. Inserting this into $h$, we obtain that

\begin{equation} \label{eq_h_s0}
\begin{split}
h\left(\frac{b}{2a}\right) &= - \frac{b^2}{4a} = - \frac{(\beta -\rho_s)^2 \Vert v^k_{S^c} \Vert_1^2 \varepsilon_k}{16 \rho_1 \Vert v^k_{S^c}\Vert_1^2 } = - \frac{ ( \beta -\rho_s)^2}{16 \rho_1 }\varepsilon_k.
\end{split}
\end{equation}
Combining this with \eqref{eq:J:monotonicity:1}, we obtain, for $t = \frac{b}{2a}$,
\begin{equation*}
\begin{split}
\mathcal{J}_{\varepsilon_{k+1}}(x^{k+1}) - \mathcal{J}_{\varepsilon_{k}}(x^k) &\leq Q_{\varepsilon_k}(x^{k+1},x^k) -  \mathcal{J}_{\varepsilon_{k}}(x^k) \\
&\leq Q_{\varepsilon_k}(x^k+t v^k,x^k) -  \mathcal{J}_{\varepsilon_{k}}(x^k) \leq - \frac{(\beta-\rho_s)^2}{16 \rho_1 }\varepsilon_k. 
\end{split}
\end{equation*}
Hence, by rearranging terms it follows that
\begin{equation}\label{eq:function_J_decreasing}
\mathcal{J}_{\varepsilon_{k+1}}(x^{k+1}) -  \onenorm{x_{*}} \le \mathcal{J}_{\varepsilon_{k}}(x^{k}) -  \onenorm{x_{*}} - \frac{(\beta-\rho_s)^2}{16 \rho_1 }\varepsilon_k. 
\end{equation}
In order to proceed, we need to bound $\varepsilon_k$ from below.
For that, we note that 
\begin{equation*}
\varepsilon_k = \min\left(\varepsilon_{k-1} , \frac{ \sigma_s(x^{k})_{\ell_1}}{N} \right) = \frac{ \sigma_s \left(x^{\ell}\right)_{\ell_1}   }{N} 
\end{equation*}
for some $\ell \le k$. By \Cref{lemma:epscontrol}, we have the following inequality chain
\begin{align*}
N\varepsilon_k =  \sigma_s \left(x^{\ell}\right)_{\ell_1} &\ge   \frac{1}{3} \left(\mathcal{J}_{\varepsilon^{\ell}} \left(x^{\ell}\right) - \onenorm{x_{*}} \right)\ge   \frac{1}{3} \left(\mathcal{J}_{\varepsilon^{k}} \left(x^{k}\right) - \onenorm{x_{*}} \right),
\end{align*}
where in the second inequality we have used that, by induction, $\mathcal{J}_{\varepsilon_k} \left(x^{k}\right) \le  \mathcal{J}_{\varepsilon^{\ell}} \left(x^{\ell}\right) $.
Plugging this into \eqref{eq:function_J_decreasing} leads to 
\begin{align*}
\mathcal{J}_{\varepsilon_{k+1}}(x^{k+1}) -  \onenorm{x_{*}} \le \left(  1-  \frac{(\beta-\rhos)^2}{48 \rho_1  N }\right)\left(     \mathcal{J}_{\varepsilon_{k}}(x^{k}) -  \onenorm{x_{*}} \right). 
\end{align*}
This finishes the induction step and concludes the proof of \Cref{thm:p1:linearrate}.

\end{proof}

\noindent From \Cref{thm:p1:linearrate} we can deduce \Cref{mainresult:sparse},  the first main result of this manuscript.

\begin{proof}[Proof of  \Cref{mainresult:sparse} ]
Recall that by \Cref{thm:p1:linearrate} we have for all $k \in \mathbb{N}$ that
\begin{align*}
\mathcal{J}_{\varepsilon_{k}}(x^{k}) -  \onenorm{x_{*}} &\le \left(  1-  \frac{c_{\rhos}}{\rho_1 N}   \right)^k \left(     \mathcal{J}_{\varepsilon_{0}}(x^0) -  \onenorm{x_{*}} \right)
\end{align*}
with a constant $ c_{\rho_s} =   \frac{(3/4-\rhos)^2}{48}   $ and where $S$ denotes the set, which contains the $s$ largest entries  of $x_{*}$ in absolute value. . By our assumption $\rho_s < 1/2$ it follows that $c_{\rho_s} \ge 1/768$, which implies that inequality \eqref{equ:linearconvergence1} holds.\\
 
\noindent By \Cref{lemma:epscontrol} we have that
\begin{align*}
  \mathcal{J}_{\varepsilon_{0}}(x^0) -  \onenorm{x_{*}}   \le 3 \sigma_s \left( x^0 \right)_{\ell_1} 
 \le 3 \onenorm{x^0-x_{*}}.
\end{align*}
Next, we note that, again by \Cref{lemma:epscontrol}, it holds that
\begin{align*}
	\frac{1-\rhos}{1+\rhos}  \onenorm{x -x_{*}} \le \mathcal{J}_{\varepsilon_{k}}(x^{k}) -  \onenorm{x_{*}}  .
\end{align*}
Combining the three inequalities in this proof together with the assumption $\rho_s \le 1/2$ yields inequality \eqref{ineq:globalconvergenerate1}, which finishes the proof.
\end{proof}
Next, we are going to prove the second main result in this manuscript, \Cref{mainresult:approximatesparse}, which deals with the approximately sparse case.
\begin{proof}[Proof of \Cref{mainresult:approximatesparse}]
Recall that
\begin{align*}
 \hat{k}:= \min \left\{ k \in \mathbb{N}: \  \sigmaxstar > \frac{2}{9}\onenorm{\left(x_{*}\right)_{S^c} -x^{k}_{S^c}}   \right\}.
 \end{align*}
Moreover, we note that by \Cref{thm:p1:linearrate} we have for $ k \le \hat{k}  $
\begin{equation}\label{ineq:intern4}
\mathcal{J}_{\varepsilon_{k}}(x^{k}) -  \onenorm{x_{*}} \le \left(  1-  \frac{c_{\rhos}}{\rho_1 N}   \right)^k \left(     \mathcal{J}_{\varepsilon_{0}}(x^0) -  \onenorm{x_{*}} \right) 
\end{equation}
with a constant $ c_{\rho_s} =      \frac{(1/4-\rhos)^2}{48}  $. Hence, by our assumption $\rho_s < 1/8$ we obtain $c_{\rho_s} \ge 1/3072$ and inequality \eqref{equ:linearconvergence2} follows, which proves the first statement. In order to prove the second statement, let $\tilde{k} $ and $k$ be natural numbers, such that $  \tilde{k}  \le \hat{k}$ and $k \ge \tilde{k} $ holds. Then we obtain that
\begin{align*}
  \frac{1-\rhos}{1+\rhos}  \onenorm{x^k -x_{*}}  -  2 \sigmaxstar &\le \mathcal{J}_{\varepsilon_{k}}(x^{k}) -  \onenorm{x_{*}}\\
 &\le \mathcal{J}_{\varepsilon_{\tilde{k} }}(x^{\tilde{k} }) -  \onenorm{x_{*}}\\
&\le \left(  1-  \frac{c_{\rhos}}{\rho_1 N}   \right)^{\tilde{k} } \left(     \mathcal{J}_{\varepsilon_{0}}(x^0) -  \onenorm{x_{*}} \right)  \\
&\le  3 \left(  1-  \frac{c_{\rhos}}{\rho_1 N}   \right)^{\tilde{k} }   \sigma_s \left(x^0\right)_{\ell_1},
\end{align*}
where in the first inequality we applied \Cref{lemma:epscontrol}. In the second inequality we used that the sequence $ \left\{ \mathcal{J}_{\varepsilon^{\ell}}  \left(  x^{\ell}  \right) \right\}_{\ell} $ is monotonically decreasing and in the third inequality we used inequality \eqref{ineq:intern4}. In the fourth inequality we again used \Cref{lemma:epscontrol}.
By rearranging terms and using the assumption $ \rhos < 1/8 $ it follows for all integers $\tilde{k}$ and $k$ such that   $\tilde{k}  \le  \hat{k}$ and $k \ge \tilde{k}$ 
\begin{align}\label{ineq:intern18}
  \onenorm{x^k -x_{*}}    &\le  6 \left(  1-  \frac{c_{\rhos}}{\rho_1 N}   \right)^{\tilde{k} }   \sigma_s \left(  x^{ 0 } \right)_{\ell_1}+ 4  \sigmaxstar    .
\end{align}
In order to proceed,  recall that $S$ denotes the support of the $s$ largest entries of $x_{*}$. Then we note that
\begin{equation}\label{ineq:intern2}
\sigma_s \left(x^0\right)_{\ell_1} \le   \onenorm{x^0_{S^c}} \le   \onenorm{\left(x^0  -x_{*}\right)_{S^c} }  + \onenorm{\left(x_{*}\right)_{S^c} } \le  \onenorm{ x^0 -x_{*} }  + \sigmaxstar.
\end{equation}
Hence, we have shown that for all integers $\tilde{k}$ and $k$ such that   $\tilde{k}  \le \hat{k}$ and $k \ge \tilde{k}$ it holds that
\begin{equation}\label{ineq:intern1}
\onenorm{x^k -x_{*}}    \le  6 \left(  1-  \frac{c_{\rhos}}{\rho_1 N}   \right)^{\tilde{k} }  \onenorm{x^0 -x_{*}} + 10 \sigmaxstar.
\end{equation}
By setting $k=\tilde{k} $, we observe that this implies inequality \eqref{ineq:approxsparse1}, which shows the second statement. In order to prove the third statement, we will distinguish two cases. For the first case, assume that $ \hat{k} \ge \Big\lceil \frac{\rho_1 N }{c_{\rho_s}} \log \left(   \frac{  \onenorm{x^0 -x_{*}} }{  \sigmaxstar }   \right) \Big\rceil $.
Then for $ k \ge \tilde{k} :=  \Big\lceil  \frac{\rho_1 N }{c_{\rho_s}} \log \left(   \frac{  \onenorm{x^0 -x_{*}} }{  \sigmaxstar }    \right) \Big\rceil  $ it follows from inequality \eqref{ineq:intern1} that
\begin{align*}
\onenorm{x^k -x_{*}}    & \le  6 \left(  1-  \frac{c_{\rhos}}{\rho_1 N}   \right)^{\frac{\rho_1 N }{c_{\rho_s}} \log \left(   \frac{  \onenorm{x^0 -x_{*}} }{  \sigmaxstar }  \right)  }  \onenorm{x^0 -x_{*}} + 10 \sigmaxstar \le  20 \sigmaxstar.
\end{align*}
where in the second inequality we have used the elementary inequality $\log \left(1+t\right) \le t$ for $t>-1$. This shows the third statement in the first case. To prove the second case, assume that $ \hat{k} < \Big\lceil \frac{\rho_1 N }{c_{\rho_s}} \log \left(   \frac{ \onenorm{x^0 -x_{*}} }{  \sigmaxstar }   \right) \Big\rceil $. Then we can compute that
\begin{align*}
\frac{1-\rhos}{1+\rhos}  \onenorm{x^k -x_{*}}  -  2 \sigmaxstar &\le \mathcal{J}_{\varepsilon_{k}}(x^{k}) -  \onenorm{x_{*}}\\
&\le \mathcal{J}_{\varepsilon_{\tilde{k} }}(x^{\hat{k} }) -  \onenorm{x_{*}}\\
&\le  3  \sigma_s \left(x^{\hat{k}}\right)_{\ell_1}\\
&\le 3 \onenorm{ x^{\hat{k}} -x_{*} }  + 3 \sigmaxstar \\
&\le 3 \left(1+\rho_s\right) \onenorm{ \left( x^{\hat{k}} -x_{*} \right)_{S^c}  }  + 3 \sigmaxstar \\
&\le 20 \sigmaxstar.
\end{align*}
In the first and third inequality we have used \Cref{lemma:epscontrol}. In the second inequality we have used the monotonicity of the sequence $ \left\{ \mathcal{J}_{\varepsilon_{k}} \left(x^k\right)  \right\}_k $. In the fourth inequality we have argued as in  inequality \eqref{ineq:intern2} and in the fifth inequality we have used the null space property. In the last inequality we have used that by definition of $\hat{k} $ it holds that  $ \sigmaxstar > \frac{2}{9}\onenorm{\left(x_{*}\right)_{S^c} -x^{\hat{k}}_{S^c}} $. This shows that also in the second case the third statement holds, which finishes the proof.
\end{proof}

\section{Computationally Efficient Weighted Least Squares Updates}\label{section:ineq_solution}
As pointed out in \Cref{section:background}, the weighted least squares step \eqref{step_1} of \Cref{alg:algo1} can be computed using the explicit formula provided by the weighted Moore-Penrose inverse such that $x^{k+1} = W_k^{-1} A^* (A W_k^{-1} A^*)^{-1} (y)$ where $W_k = \diag(w_k)$, see e.g. \cite[Chapter 15.3]{FoucartRauhut13}. This leads to the main computational step of solving the positive linear system 
\begin{equation} \label{eq:mm:linsystem}
    \left(A W_k^{-1} A^*\right) z = y
\end{equation}
of size $(m \times m)$, which is solved in practice often by an iterative solver such as the conjugate gradient method \cite{HestenesStiefel52,Meurant}, especially if the measurement matrix $A$ allows for fast matrix-vector multiplications \cite{Voronin12,fornasier_peter_rauhut_worm}. The authors of \cite{fornasier_peter_rauhut_worm} provide conditions on the accuracy of conjugate gradient (CG) solvers of the successive linear systems for an IRLS algorithm for basis pursuit using weights that are given by 
\begin{equation} \label{eq:weights:Daubechies}
    (w_{k})_i = 1/\sqrt{|x_i^k|^2+\varepsilon_k^2}
\end{equation}
for each $i \in [n]$ (coinciding with the choice of the weights of \cite{Daubechies10,AravkinBurkeHe19}).

It was observed that combining IRLS with CG provides a competitive algorithm for sparse recovery if the required  accuracy is chosen appropriately  \cite{fornasier_peter_rauhut_worm}, however, with the problem that as the IRLS method approaches a sparse solution, and therefore, the smoothing parameter $\varepsilon_k$ becomes small, the linear system matrix $A W_k^{-1} A^*$ suffers more and more from ill-conditioning \cite[Section 5.2]{fornasier_peter_rauhut_worm}. To the best of our knowledge, this numerical issue has not been resolved for any IRLS method for sparse recovery in the literature so far.

In this section, we provide an implementation of the weighted least squares step \eqref{step_1} that avoids the ill-conditioning issue of previous IRLS methods and we sketch why the condition number of the resulting linear system can indeed be bounded under standard assumptions. 

The starting point of this implementation is the observation that for weight choice \eqref{eq:IRLS:step3} used in this paper, i.e., for 
\begin{equation*}
    \left( w_{k} \right)_i :=  \frac{1}{ \max \left( \vert x_i^{k}  \vert , \varepsilon_{k}   \right) }
\end{equation*}
for each $i \in [N]$, unlike for \eqref{eq:weights:Daubechies}, it is possible to write the inverse weight matrix $W_k^{-1} \in \R^{N \times N}$ such that 
\begin{equation} \label{eq:Wkmin1:formula}
    W_k^{-1} = W_{I_k}^{-1} + \varepsilon_k \left(\text{Id}_{N} - P_{I_k}  \right)  = \left(W_{I_k}^{-1}- \varepsilon_k P_{I_k}\right) + \varepsilon_k \text{Id},
\end{equation}
if $I_k := \{i \in [N]: |x^k_i| > \varepsilon_k\}$ denotes the set $I$ of the proof of \Cref{thm:p1:linearrate}, $W_{I_k}^{-1} \in \R^{N \times N}$ is the diagonal matrix with entries $|x^k_i|$ if $i \in I_k$ and $0$ otherwise, and where $P_{I_k}$ denotes the projection matrix such that $P_{I_k}x = x_{I_k}$. In particular, we observe that $(W_k)^{-1}$ is the sum of a scaled identity matrix $\text{Id}$ and a (diagonal) matrix with only $s_k:=|I_k|$ non-zero entries. Furthermore, due to the update rule \eqref{step_2} of the smoothing parameter $\varepsilon_k$, it can be seen that $s_k$ is small and of the order of $s$ if $\varepsilon_k$ approaches $0$, i.e., if the $k$-th iterate $x^k$ of \Cref{alg:algo1} has only small coordinates outside a set of $s$ large coordinates. 

By the Sherman-Morrison-Woodbury matrix inversion formula \cite{Woodbury50}, cf. \Cref{lemma:Woodbury} below, it is therefore possible to reformulate the weighted least squares problem such that $x^{k+1}$ can be computed such that solving a positive definite linear system of size $(s_k \times s_k)$, which furthermore is well-conditioned, becomes the main computational step, avoiding solving the ill-conditioned system \eqref{eq:mm:linsystem} in this way. We summarize this implementation in \Cref{algo:IRLS:mainstep:implementation} below.

The implementation uses the matrix $V \in \R^{m \times N}$ with orthonormal columns which denotes a the projection onto the range space of the measurement matrix $A \in \R^{m \times N}$, as well as the left singular vector matrix $U \in \R^{m \times m}$ of $A$ and the diagonal matrix $\Sigma_A \in \R^{m \times m}$ containing the singular values of $A$. These can be pre-computed before using IRLS, for example via a singular value decomposition of $A$. \footnote{In a large scale setting where $A$ is only accessed through matrix-vector multiplications, providing this information is not necessary as respective information can be replaced by using matrix-vector multiplications with $A$, $(A A^*)^{-1}$ and $A^*$.} Likewise, the vector 
\begin{equation} \label{eq:def:ytilde}
    \widetilde{y} = V \Sigma_A^{-1} (U^* y)
\end{equation}
can be pre-computed and can be re-used at each outer iteration of IRLS.

In the following, if $I \subset [N]$, we denote by $M_I \in \R^{m \times |I|}$ the restriction of a matrix $M \in \R^{m \times N}$ to the columns indexed by $I$, and by $Q_{I_k} \in \R^{N \times I_k} $ the projector matrix such that $P_{I_k} = Q_{I_k} Q_{I_k}^*$. Furthermore, let $D_{I_k}^{-1} \in \R^{I_k \times I_k}$ be a diagonal matrix such that $(D_{I_k}^{-1})_{ii} = |x^k_i|$ for each $i \in I_k$. 

\begin{algorithm}[h]
\caption{Practical implementation of weighted least squares step of IRLS for small $\varepsilon_k$} \label{algo:IRLS:mainstep:implementation}
\begin{algorithmic}
\STATE{\bfseries Input:} Matrix $V \in \mathbb{R}^{m \times N}$ projecting onto range space of measurement matrix $A$, $\widetilde{y} \in \mathbb{R}^{N}$ from \eqref{eq:def:ytilde}, smoothing parameter $\varepsilon_k$, projection $\gamma_k^{(0)} = Q_{I_k}^*Q_{I_{k-1}}(\gamma_{k-1}) \in \mathbb{R}^{I_k}$ of solution $\gamma_{k-1} \in \R^{I_{k-1}}$ of linear system \eqref{eq:gamma:system} for previous iteration $k-1$.
\end{algorithmic}
\begin{algorithmic}[1]
\STATE Compute $h_k^0 = Q_{I_k}^*\widetilde{y} - \left(\varepsilon_k \left(D_{I_k}^{-1}- \varepsilon_k \text{Id}\right)^{-1} + (V^*)_{I_k}^* (V^*)_{I_k}\right)  \gamma_k^{(0)} \in \R^{I_k}$.
\STATE Solve 	
\begin{equation} \label{eq:gamma:system} 
	\left(\varepsilon_k \left(D_{I_k}^{-1}- \varepsilon_k \text{Id}\right)^{-1} + (V^*)_{I_k}^* (V^*)_{I_k}\right) \Delta\gamma_{k} =  h_k^0	
	\end{equation}
	 for $\gamma_k \in \R^{I_k}$ by the \emph{conjugate gradient} method \cite{HestenesStiefel52,Meurant}.
 \STATE Compute $\gamma_k = \gamma_k^{(0)} + \Delta\gamma_k \in \R^{I_k}$.
\STATE Compute residual $r_{k+1}:= \widetilde{y}-V (V^*)_{I_k} (\gamma_{k}) \in \R^N$.
\STATE Set $x^{k+1} = r_{k+1}$.
\STATE Set $(x^{k+1})_{I_k} = (x^{k+1})_{I_k} + \gamma_k$.
\STATE{\bfseries Output:} $x^{k+1} \in \R^N$ and $\gamma_k \in I_k$.
\end{algorithmic}
\end{algorithm}

Next, we verify that $x^{k+1}$ as computed by \Cref{algo:IRLS:mainstep:implementation} indeed is a solution of the weighted least squares problem \eqref{step_1}. 

As a preparation, we state the Sherman-Morrison-Woodbury formula for matrix inversion.
\begin{lemma}[{\cite{Woodbury50}}] \label{lemma:Woodbury}
Let $B \in \mathbb{R}^{n \times n}, C \in \mathbb{R}^{k \times k}, E \in \mathbb{R}^{n \times k}$ and $F \in \mathbb{R}^{k \times n}$. Then, it holds that 
\begin{equation*} \label{eq_Woodbury_identity}
(E C F^* + B)^{-1} = B^{-1} - B^{-1} E ( C^{-1} + F^* B^{-1} E )^{-1} F^* B^{-1}
\end{equation*}
\end{lemma}

With this, we proceed to the proof of the following statement.
\begin{lemma} \label{lemma:weightedleastsquares:implement}
If $x^{k+1} \in \R^N$ is the output of \Cref{algo:IRLS:mainstep:implementation}, then $x^{k+1}$ coincides with the solution of the weighted least squares problem \eqref{step_1}. 
\end{lemma}

\begin{proof}[{Proof of \Cref{lemma:weightedleastsquares:implement}}]
    We recall from above that if $x_*^{k+1} \in \R^{N}$ is the solution of \eqref{step_1}, it holds that  $x_*^{k+1} = W_k^{-1} A^*(z)$ where $z$ is as in \eqref{eq:mm:linsystem}. Using \eqref{eq:Wkmin1:formula}, we see that
\begin{equation*}
    A W_k^{-1} A^* = A \left(\left(W_{I_k}^{-1}- \varepsilon_k P_{I_k}\right) + \varepsilon_k \text{Id}\right) A^* = A_{I_k} \left(D_{I_k}^{-1}- \varepsilon_k \text{Id}\right)A_{I_k}^* + \varepsilon_k A A^*.
\end{equation*}

By identifying $B:= \varepsilon_k AA^*$, $C := \left(D_{I_k}^{-1}- \varepsilon_k \text{Id}\right)$ and $E= F := A_{I_k}$, and by noting that the matrix $C$ is invertible, since, on the set $I_k$, we have $|x^{k+1}_i|> \varepsilon_k $, we obtain by using \Cref{lemma:Woodbury} that
\begin{equation*} 
\left(A(W_k)^{-1} A^*\right)^{-1}= \varepsilon_k^{-1} Z - \varepsilon_k^{-1} Z A_{I_k}G^{-1} A_{I_k}^* Z,
\end{equation*}
using the notation $Z: =(A A^*)^{-1}$ and $G: = \varepsilon_k C^{-1} + A_{I_k}^* Z A_{I_k}$.

Therefore, denoting the matrix $ \varepsilon_k C^{-1} + A_{I_k}^* Z A_{I_k}$ by $G$, we have

\begin{equation} \label{eq:z:maineq}
z=\left(A(W_k)^{-1} A^*\right)^{-1}y = \varepsilon_k^{-1} Z(y) - \varepsilon_k^{-1} Z A_{I_k}G^{-1} A_{I_k}^*Z(y) =\varepsilon_k^{-1}Z(y- A_{I_k}G^{-1} A_{I_k}^*Z(y))
\end{equation}
and therefore 

\begin{equation} \label{eq:AIkz}
\begin{split}
A_{I_k}^*(z) &= \varepsilon_k^{-1} (A_{I_k}^*Z(y)- A_{I_k}^* Z A_{I_k}G^{-1} A_{I_k}^*Z(y)) \\
&= \varepsilon_k^{-1} (A_{I_k}^*Z(y)- A_{I_k}^* Z A_{I_k}\left(\varepsilon_k C^{-1} + A_{I_k}^* Z A_{I_k}\right)^{-1} A_{I_k}^*Z(y)) \\
 &=\varepsilon_k^{-1} \left(A_{I_k}^*Z(y)-  \left(A_{I_k}^* Z A_{I_k}  \pm \varepsilon_k C^{-1}  \right) \left(\varepsilon_k C^{-1} + A_{I_k}^* Z A_{I_k}\right)^{-1} A_{I_k}^*Z(y)\right)  \\
&= C^{-1}  \left(\varepsilon_k C^{-1} + A_{I_k}^*Z A_{I_k}\right)^{-1} A_{I_k}^*Z(y).
\end{split}
\end{equation}

Thus, for the $(k+1)$-th iteration of IRLS, we obtain for the solution $x_{*}^{k+1}$ of \eqref{step_1} the representation 
\begin{equation} \label{eq:Z:calculation:computational}
\begin{split}
x_{*}^{k+1} &= W_k^{-1}A^*(z) =  \left(\varepsilon_k \Id + Q_{I_k} \left( D_{I_k}^{-1} - \varepsilon_k \Id \right) Q_{I_k}^*  \right)A^*(z) \\
&= \left[ \varepsilon_k \text{Id} + Q_{I_k} C Q_{I_k}^*  \right]A^* z = \varepsilon_k A^* z +   Q_{I_k} C A_{I_k}^*(z) \\
&= \varepsilon_k A^* z + Q_{I_k} C C^{-1} \left(\varepsilon_k C^{-1} + A_{I_k}^* Z A_{I_k}\right)^{-1} A_{I_k}^*Z(y) \\
&= \varepsilon_k A^* z + Q_{I_k} G^{-1} A_{I_k}^*Z(y)
\end{split}
\end{equation}
using \eqref{eq:AIkz} in the third equality.

Next, if $V$ is as in the input of \Cref{algo:IRLS:mainstep:implementation}, since
\[
A_{I_k}^* Z A_{I_k} = Q_{I_k}^* A^* (A A^*)^{-1} A Q_{I_k} = Q_{I_k}^* V V^* Q_{I_k} = (V^*)_{I_k}^* (V^*)_{I_k},
\]
we observe that the matrix $G$ from above actually coincides with the linear system matrix of \eqref{eq:gamma:system}. 

Furthermore, if $\gamma_k$ is as in step 3 of \Cref{algo:IRLS:mainstep:implementation}, it satisfies
\[
\gamma_k = \gamma_k^{(0)} + \Delta\gamma_k =  \gamma_k^{(0)} + G^{-1}\f{h}_k^0 = \gamma_k^{(0)} + G^{-1} \left( Q_{I_k}^*\widetilde{y}  - G \gamma_k^{(0)}\right) =  G^{-1}Q_{I_k}^*\widetilde{y},
\]
where we also observe that
\[
Q_{I_k}^*\widetilde{y} = Q_{I_k}^* V \Sigma_A^{-1}U^*(y) = Q_{I_k}^* A^* \left(A A^*\right)^{-1} (y) = A_{I_k}^* Z (y).
\]
Using the latter equation, we can identify $z$ of \eqref{eq:z:maineq} such that
\[
z = \varepsilon_k^{-1}Z(y- A_{I_k}G^{-1} A_{I_k}^*Z(y)) = \varepsilon_k^{-1}Z\left(y - A_{I_k}G^{-1} Q_{I_k}^*\widetilde{y}\right) = \varepsilon_k^{-1}Z\left(y - A_{I_k}\gamma_k\right).
\]
Inserting this into \eqref{eq:Z:calculation:computational}, we obtain
\[
\begin{split}
x_{*}^{k+1} &= \varepsilon_k A^* z + Q_{I_k} G^{-1} A_{I_k}^*Z(y) = \varepsilon_k A^* z + Q_{I_k} G^{-1} Q_{I_k}^*\widetilde{y} = \varepsilon_k A^* z + Q_{I_k} \gamma_k \\
&=  A^* Z\left(y - A_{I_k}\gamma_k\right) + Q_{I_k} \gamma_k \\
&= V \Sigma_A^{-1} U^*\left(y - U \Sigma_A V^*Q_{I_k}\gamma_k\right) + Q_{I_k} \gamma_k \\
&= \widetilde{y} - V (V^*)_{I_k}\gamma_k + Q_{I_k} \gamma_k  \\
&= r_{k+1} + Q_{I_k} \gamma_k,
\end{split}
\]
where the residual $r_{k+1}$ is as in step 4 of \Cref{algo:IRLS:mainstep:implementation}. 

Comparing the equation $x_{*}^{k+1} = r_{k+1} + Q_{I_k} \gamma_k$ with the steps 5 and 6 of \Cref{algo:IRLS:mainstep:implementation}, we observe that the output $x^{k+1}$ of \Cref{algo:IRLS:mainstep:implementation} coincides with $x_{*}^{k+1}$, which finishes the proof.

\end{proof}
It is well-known that if the conjugates gradient method is used as an inexact solver using a limited number of iterations only, a sufficient condition for quantifying its accuracy can be achieved by bounding the condition number of the system matrix (see, e.g., \cite[Section 5.1]{NocedalWright06}). 

In fact, using few iterations of a CG method to solve the system \eqref{eq:gamma:system} leads typically to quite accurate solutions, in particular if $\varepsilon_k$ is small. This can be seen by analyzing the condition number of the matrix $G$
\begin{equation} \label{eq:systemmatrix:good}
G=\varepsilon_k \left(D_{I_k}^{-1}- \varepsilon_k \text{Id}_{I_k}\right)^{-1} + A_{I_k}^* Z A_{I_k} =: M_{1,k} + M_{2,k}
\end{equation}
of \cref{eq:gamma:system} in \Cref{algo:IRLS:mainstep:implementation}. 

We start by bounding the condition number of $M_{2,k}$. For that, let $x$ such that $ \twonorm{x}=1 $. We calculate that
\begin{equation*}
    \Vert x^T A_{I_k}^*(A A^*)^{-1} A_{I_k} x \Vert \le \Vert A_{I_k} x  \Vert^2 \Vert (A A^*)^{-1}  \Vert  
    \le (1+\delta)^2 \Vert (A A^*)^{-1} \Vert,
\end{equation*}
where in the last inequality we have assumed that $A_{I_k}$ satisfies the restricted isometry property. This implies that
\begin{equation}\label{ineq:aux1}
\Vert M_{1,k} \Vert =\Vert  A_k^*(A A^*)^{-1} A_k  \Vert \le (1+\delta)^2 \sigma_{\min} (A)^{-2}.
\end{equation}
In a similar way we can derive that
\begin{equation}\label{ineq:aux2}
\sigma_{\min} \left(  A_k^*(A A^*)^{-1} A_k   \right) \ge (1-\delta)^2 \Vert A \Vert^{-2}.    
\end{equation}
We note that $M_{1,k}$ is a diagonal matrix with entries given by $\varepsilon_k/(x^k_i-\varepsilon_k)$. Now, by additionally assuming that the iteration $x_k$ is already close to the ground truth and, without loss of generality, by assuming that $k$ is such that $||x^k-x_*||_{\infty}\leq 1/4||x_*||_{\infty}$, we have

\begin{equation}\label{ineq:aux3}
||M_{1,k}||_{2}= \varepsilon_k/(x_i^k-\varepsilon_k)\leq \varepsilon_k/(\frac{3}{4} \min_i \vert {(x_{\ast})}_i\vert  -\varepsilon_k).
\end{equation}
Note that this term becomes arbitrarily small, when $\varepsilon_k$ becomes arbitrarily small.

Hence, for $\varepsilon_k$ small enough it follows that
\begin{equation*}
\kappa \left(G \right) = \frac{\Vert G \Vert}{ \sigma_{\min} (G) } \overset{(a)}{\le} \frac{\Vert M_{1,k}\Vert +\Vert M_{2,k}\Vert  }{ \sigma_{\min} (M_{2,k})  - \Vert M_{1,k}\Vert}
\overset{(b)}{\le}    \frac{ 2 (1+\delta)^2 \Vert A \Vert^{2}}{ (1-\delta)^2 \sigma_{\min} (A)^{2}} = \frac{2(1+\delta)^2 }{(1-\delta)^2} \kappa \left( A \right)^2  .   
\end{equation*}
Here, in (a) we used Weyl's inequality and (b) holds as soon as $\varepsilon_k$ is small enough due to inequalities \eqref{ineq:aux1}, \eqref{ineq:aux2}, and \eqref{ineq:aux3}. Hence, we have shown that if $A$ is well-conditioned, then also the matrix $G$ will be well-conditioned and the CG-method will yield very accurate solutions.

\section{Proofs of technical lemmas}\label{sec:proof_technical_lemmas}

\subsection{Proof of \Cref{lemma:IRLS:basicproperties}} \label{sec:appendix:proofbasiclemma}

In this section, we prove that the chosen quadratic upper bound satisfies the majorization-minimization property.

\begin{proof}
We prove each of the three statements separately. 
\begin{enumerate}
	\item
	Let $x \in \R^N$. Then the $i$-th coordinate of $\diag(w_{\varepsilon}(x)) x$ is given by
	\[
	(\diag(w_{\varepsilon}(x)) x)_i = \begin{cases}
	\frac{x_i}{|x_i|} = \sgn(x_i), & \text{ if } |x_i| > \varepsilon, \\
	\frac{x_i}{\varepsilon}, & \text{ if } |x_i| \leq \varepsilon.
	\end{cases} = j'_{\varepsilon}(x_i)  = \left( \nabla \mathcal{J}_{\varepsilon}(x) \right)_i,
	\]
	where $\mathcal{J}_{\varepsilon}(x)$ is the gradient of $\mathcal{J}_{\varepsilon}$ at $x$.
	\item  This follows directly from the definition of $ Q_{\varepsilon}(x,z)$ and by setting $x=z$.
	\item We define $I := \{i \in [N]: |x_i| > \varepsilon \}$ and write the difference $Q_{\varepsilon}(z,x) - \mathcal{J}_{\varepsilon}(z)$ as
	\begin{equation*}
	\begin{split}
	Q_{\varepsilon}(z,x)\! -\! \mathcal{J}_{\varepsilon}(z) &= \frac{1}{2}\left(\langle z, \diag(w_{\varepsilon}(x)) z\rangle - \langle x,  \diag(w_{\varepsilon}(x)) x \rangle  \right) \\
	&=\sum_{i \in I} \left(\frac{1}{2}|x_i| + \frac{1}{2} \frac{z_i^2}{|x_i|} - j_{\varepsilon}(z_i)\right) +  \sum_{i \in I^c} \left(\frac{1}{2} \varepsilon + \frac{1}{2} \frac{z_i^2}{\varepsilon} - j_{\varepsilon}(z_i)\right)
	\end{split}
	\end{equation*}
	and show that each summand of the two sums is non-negative. In particular, if $i \in I$, then assume first that $|z_i| > \varepsilon$. Then 
	\[
	\frac{1}{2}|x_i| + \frac{1}{2} \frac{z_i^2}{|x_i|} - j_{\varepsilon}(z_i) = \frac{1}{2}\left(|x_i| + \frac{z_i^2}{|x_i|}\right) - |z_i| \geq |z_i| - |z_i| = 0
	\]
	due to inequality $a \leq \frac{1}{2}(a^2/b + b)$, which holds	 for any $b > 0$.
	
	On the other hand, if $|z_i| \leq \varepsilon$, then
	\begin{equation}
	\begin{split}
	\frac{1}{2}|x_i| + \frac{1}{2} \frac{z_i^2}{|x_i|} - j_{\varepsilon}(z_i) &= \frac{1}{2}|x_i| + \frac{1}{2} \frac{z_i^2}{|x_i|} - \frac{1}{2}\left(\frac{z_i^2}{\varepsilon}+\varepsilon\right)\\ 
	& = \frac{1}{2}\big( |x_i| - \varepsilon\big) + \frac{1}{2} z_i^2\left(\frac{1}{|x_i|}-\frac{1}{\varepsilon}\right) \\
	&\ge \frac{1}{2}\big( |x_i| - \varepsilon\big) + \frac{1}{2} \varepsilon^2 \left(\frac{1}{|x_i|}-\frac{1}{\varepsilon}\right) \\
	&=  \frac{1}{2}\left(|x_i| +  \frac{\varepsilon^2}{|x_i|}\right) - \varepsilon \geq \varepsilon - \varepsilon = 0,
	\end{split}
	\end{equation}
	where we used that $\frac{1}{|x_i|}-\frac{1}{\varepsilon} < 0$ in the first inequality. In the second inequality, we again used $a \leq \frac{1}{2}(a^2/b + b)$  for any $b > 0$.  Now let $i \in I^c$. We again consider the two cases, $|z_i| \leq \varepsilon$ and $|z_i| > \varepsilon$. In the first case we have that $\frac{1}{2}\varepsilon + \frac{1}{2} \frac{z_i^2}{\varepsilon} - j_{\varepsilon}(z_i) = 0$, and in the second case we have that
	\begin{equation*}
		\frac{1}{2}\varepsilon + \frac{1}{2} \frac{z_i^2}{\varepsilon} - j_{\varepsilon}(z_i) = \frac{1}{2}\varepsilon + \frac{1}{2} \frac{z_i^2}{\varepsilon} - |z_i| \ge  |z_i| - |z_i| = 0,
	\end{equation*}	
which concludes the proof.
\end{enumerate}
\end{proof}

\subsection{Proof of \Cref{prop:DDFG:supportcondition}} \label{appendix:proof:supportcondition}

In this section, we prove \Cref{prop:DDFG:supportcondition}, namely, if $x^k$ , k-th IRLS iteration, is already close enough from the ground truth, then we would observe that the support would have been already identified and, consequently, the hardest part of the sparse recovery problem would have been solved.

\begin{proof}
Let $j \in S^c$, where $S$ is the support set of $x_{*}$. Then
\[
	|x_j^k| \leq \sum_{i \in S^c} |x_i^k| < \min_{ i \in S} |(x_*)_{i}| - \sum_{i \in S} |x_i^k- (x_*)_i|,
\]
using the assumption $\sum_{i \in S^c} |(x^k)_{i}| + \sum_{i \in S} |x_i^k-(x_*)_{i}| = \|x^k - x_{*}\|_1 < \min_{ i \in S} |(x_*)_{i}|$. 

On the other hand, for $j \in S$, we can estimate that
\[
|x_j^k| = |x_j^k - (x_*)_j + (x_*)_j| \geq |(x_*)_j| - |x_j^k - (x_*)_j| \geq \min_{i \in S} |(x_*)_i| - \sum_{i \in S} |x_i^k - (x_*)_i|.
\]
Taking the previous two inequalities together, we conclude that 
\[\max_{j \in S^c} |x_j^k| < \min_{j \in S} |x_j^k|,
\]
which finishes the proof.
\end{proof}

\end{document}